\documentclass[11pt]{amsart}
\usepackage{amsmath,amscd,amssymb,amsfonts,amsthm}

 \newtheorem{theorem}{Theorem}[section]

\newtheorem{proposition}[theorem]{Proposition}

\newtheorem{lemma}[theorem]{Lemma}
\theoremstyle{remark}
\newtheorem{remark}[theorem]{Remark}
\newtheorem{definition}[theorem]{Definition}

\newtheorem{example}[theorem]{Example}
\newtheorem{examples}[theorem]{Examples}
\newcommand\A{\mathbb{A}}

\renewcommand\L{\mathcal{L}}

\newcommand{\V}{\mathbb{V}}
\newcommand{\TM}{\mathbb{T}M}

\newcommand{\TS}{\mathbb{T}S}
\newcommand{\TG}{\mathbb{T}G}

\newcommand{\T}{\mathbb{T}}

\newcommand{\n}{\mathfrak{n}}

\newcommand{\X}{\mf{X}}

\newcommand{\C}{\mathbb{C}}

\newcommand{\pr}{\on{pr}}

\newcommand\lie[1]{\mathfrak{#1}}
\renewcommand{\k}{\lie{k}}
\newcommand{\h}{\lie{h}}

\newcommand{\g}{\lie{g}}

\newcommand{\on}{\operatorname}

\newcommand{\Ad}{ \on{Ad} }

\renewcommand{\ker}{ \on{ker}}

\newcommand{\Mult}{{\on{Mult}}}
\newcommand{\Inv}{{\on{Inv}}}

\newcommand{\diag}{{\on{diag}}}

\newcommand{\da}{\dasharrow}

\newcommand{\red}{{ \on{red} }}

\newcommand\dirac{/\kern-1.2ex\partial} 
\newcommand\qu{/\kern-.7ex/} 

\renewcommand\a{\mathsf{a}}

\newcommand{\hra}{\hookrightarrow}
\newcommand{\xra}{\xrightarrow}

\renewcommand{\d}{{\mbox{d}}}
\newcommand{\dd}{\mf{d}}

\newcommand{\ol}{\overline}

\newcommand\sig{\sigma}

\newcommand\Om{\Omega}

\newcommand{\f}{\frac}

\renewcommand{\l}{\langle}
\renewcommand{\r}{\rangle}
\newcommand{\hh}{{\textstyle \f{1}{2}}}
\newcommand{\ti}{\tilde}

\newcommand\pt{\on{pt}}

\newcommand{\ran}{\on{ran}}

\newcommand\rk{\on{rk}}

\newcommand\beqn{\begin{equation}}
\newcommand\eeqn{\end{equation}}

\newcommand{\mf}{\mathfrak}
\newcommand{\beq}{\begin{eqnarray*}}
\newcommand{\eeq}{\end{eqnarray*}}

\newcommand{\Cour}[1]      {[\![#1]\!]}

\begin{document}

\title[]{Courant algebroids and Poisson geometry}

\author{D. Li-Bland}
\address{University of Toronto, Department of Mathematics,
40 St George Street, Toronto, Ontario M4S2E4, Canada }
\email{dbland@math.toronto.edu}

\author{E. Meinrenken}
\address{University of Toronto, Department of Mathematics,
40 St George Street, Toronto, Ontario M4S2E4, Canada }
\email{mein@math.toronto.edu}

\date{\today}

\begin{abstract}
  Given a manifold $M$ with an action of a quadratic Lie algebra
  $\dd$, such that all stabilizer algebras are co-isotropic in $\dd$,
  we show that the product $M\times\dd$ becomes a Courant algebroid
  over $M$.  If the bilinear form on $\dd$ is split, the choice of
  transverse Lagrangian subspaces $\g_1,\g_2$ of $\dd$ defines a
  bivector field $\pi$ on $M$, which is Poisson if $(\dd,\g_1,\g_2)$
  is a Manin triple. In this way, we recover the Poisson structures of
  Lu-Yakimov, and in particular the Evens-Lu Poisson structures on the
  variety of Lagrangian Grassmannians and on the de Concini-Procesi
  compactifications. Various Poisson maps between such examples are
  interpreted in terms of the behaviour of Lagrangian splittings under
  Courant morphisms.
\end{abstract}

\maketitle
\setcounter{tocdepth}{3}

{\small \tableofcontents \pagestyle{headings}}

\setcounter{section}{-1}
\vskip.3in\section{Introduction}\label{sec:intro}
Let $\dd$ be a (real or complex) Lie algebra, equipped with an
invariant symmetric bilinear form of split signature. A pair
$\g_1,\g_2$ of transverse Lagrangian subalgebras of $\dd$ defines a
\emph{Manin triple} $(\dd,\g_1,\g_2)$. Manin triples were introduced
by Drinfeld \cite{dr:qu}, and are of fundamental importance in his
theory of Poisson Lie groups and Poisson homogeneous spaces. In their
paper \cite{ev:on}, S. Evens and J.-H. Lu found that every Manin
triple defines a Poisson structure on the variety $X$ of all
Lagrangian subalgebras $\dd$. If $\g$ is a complex semisimple Lie
algebra, and $\dd=\g\oplus\ol{\g}$ (where $\g$ carries the Killing
form and $\ol{\g}$ indicates the same Lie algebra with the opposite
bilinear form), then one of the irreducible components of $X$ is the
de Concini-Procesi `wonderful compactification' of the adjoint group
$G$ integrating $\g$. More recently, J.-H. Lu and M. Yakimov
\cite{lu:reg} studied Poisson structures on homogeneous spaces of the
form $D/Q$, where $D$ is a Lie group integrating $\dd$, and $Q$ is a
closed subgroup whose Lie algebra is co-isotropic for the inner
product on $\dd$. Their results show that if $M$ is a $D$-manifold 
such that all stabilizer algebras are co-isotropic, then the Manin
triple $(\dd,\g_1,\g_2)$ defines a Poisson structure on $M$. 

In this paper, we put this construction into the framework of Courant algebroids.
Suppose $M$ is a manifold with a Lie algebra action $\a\colon
M\times\dd\to TM$. We show that the trivial bundle $M\times\dd$
carries the structure of a Courant algebroid, with $\a$ as the anchor
map, and with Courant bracket extending the Lie bracket on constant
sections, \emph{if and only if} all stabilizer algebras are co-isotropic.  If
this is the case, any Lagrangian subalgebra of $\dd$ defines a Dirac
structure, and Manin triples $(\dd,\g_1,\g_2)$ define pairs of
transverse Dirac structures. In turn, by a result of Mackenzie-Xu
\cite{mac:lie}, any pair of transverse Dirac structures determines a
Poisson bivector $\pi$. In this way, we recover the Lu-Yakimov 
Poisson bivector. Various Poisson maps between
examples of this type are interpreted in terms of the behaviour under
Courant morphisms.

More generally, we give an explicit formula for the Schouten bracket
$[\pi,\pi]$ for the bivector resulting from any pair of transverse
Lagrangian subspaces (not necessarily Lie subalgebras). If $\g$ is a
Lie algebra with invariant inner product, then $\dd=\g\oplus\ol{\g}$
with Lagrangian splitting given by the diagonal $\g_\Delta=\{(x,x)|\ 
x\in\g\}$ and the anti-diagonal $\g_{-\Delta}=\{(x,-x)|\ x\in\g\}$ is
a typical example. The obvious $\dd$-action on the Lie group $G$
integrating $\g$ has co-isotropic (in fact Lagrangian) stabilizers,
and the bivector defined by $(\dd,\g_\Delta,\g_{-\Delta})$ is the
quasi-Poisson structure on $G$ described in \cite{al:ma,al:qu}.  In a
similar fashion, one obtains a natural quasi-Poisson structure on the
variety of Lagrangian subalgebras of $\g\oplus\ol{\g}$.

Our theory involves some general constructions with Courant
algebroids, which may be of independent interest. Suppose $\A\to N$ is
a given Courant algebroid. We introduce the notion of an \emph{action
  of a Courant algebroid} $\A\to N$ on a manifold $M$, similar to the
Higgins-Mackenzie definition of a Lie algebroid action \cite{hig:alg}.
A Courant algebroid action is given by two maps $\Phi\colon M\to N,\ 
\varrho\colon \Phi^*\A\to TM$ with a suitable compatibility condition,
and we show that the vector bundle pull-back $\Phi^*\A$ acquires the
structure of a Courant algebroid, provided all `stabilizers'
$\ker(\varrho_m)$ are co-isotropic.

The organization of this article is as follows. Section
\ref{sec:prelim} gives a review of Courant algebroids $\A\to M$, Dirac
structures, and Courant morphisms. In Section \ref{sec:con} we discuss
coisotropic reductions of Courant algebroids, which we then use to
define pull-backs $\Phi^!\A$ under smooth maps. Such a pull-back is
different from the vector bundle pull-back $\Phi^*\A$, and to define a
Courant algebroid structure on the latter we need a Courant algebroid
action $(\Phi,\varrho)$ with co-isotropic stabilizers. In Section
\ref{sec:man} we consider Lagrangian sub-bundles of Courant
algebroids.  We show that the Courant tensors of Lagrangian
sub-bundles behave naturally under `backward image', and give a simple
construction of the bivector $\pi$ for a Lagrangian splitting as the
backward image under a `diagonal morphism'. We give a formula for the
rank of $\pi$ and compute its Schouten bracket $[\pi,\pi]$. Finally,
we present a compatibility condition for Lagrangian splittings
relative to Courant morphisms, which guarantees that the underlying
map of manifolds is a bivector map.  Section \ref{sec:cou} specializes
the theory to Courant algebroids coming from actions of quadratic Lie
algebras, and puts the examples mentioned above into this framework.
In Section \ref{sec:split} we consider the Lagrangian splittings of
$M\times\dd$ coming from Lagrangian splittings of $\dd$. In the final
Section \ref{sec:poi} we show how to interpret the basic theory of
Poisson Lie groups from the Courant algebroids perspective.

\vskip.2in \noindent {\bf Acknowledgements.} We would like to thank
Jiang-Hua Lu for discussions, and for detailed comments on a first
version of this paper. We also thank Henrique Bursztyn, Yvette
Kosmann-Schwarzbach, and Pavol {\v{S}}evera for several helpful
remarks.  D.L.-B. was supported by an NSERC CGS-D Grant, and E.M. by
an NSERC Discovery Grant and a Steacie Fellowship.

\vskip.3in\section{Courant algebroids} \label{sec:prelim} 

\subsection{Definition of a Courant algebroid}\label{subsec:def}  
The notion of a \emph{Courant algebroid} was introduced by
Liu-Weinstein-Xu \cite{liu:ma} to provide an abstract framework for
Courant's theory of Dirac structures \cite{cou:di,dor:dir}. The original
definition was later simplified by Roytenberg \cite{roy:co} (based on
ideas of {\v{S}}evera \cite{sev:let}) and Uchino \cite{uch:rem} to the
following set of axioms. See Kosmann-Schwarzbach \cite{kos:qua} for a
slightly different version.

\begin{definition}
  A \emph{Courant algebroid} over a manifold $M$ is a quadruple
  $(\A,\,\l\cdot,\cdot\r,\,\mathsf{a},\,\Cour{\cdot,\cdot})$,
  consisting of a vector bundle $\A\to M$, a
  non-degenerate bilinear form (\emph{inner product})
  $\l\cdot,\cdot\r$ on the fibers of $\A$, a bundle map $\a\colon \A\to TM$
  called the \emph{anchor}, and a bilinear \emph{Courant bracket}
  $\Cour{\cdot,\cdot}$ on the space $\Gamma(\A)$ of sections, such
  that the following three axioms are satisfied:
\begin{enumerate}
\item[c1)] $\Cour{x_1,\Cour{x_2,x_3}}=\Cour{\Cour{x_1,x_2},x_3}
+\Cour{x_2,\Cour{x_1,x_3}}$, 
\item[c2)] $\a(x_1)\l x_2,x_3\r=\l \Cour{x_1,x_2},\,x_3\r+\l x_2,\,\Cour{x_1,x_3}\r$,
\item[c3)] $\Cour{x_1,x_2}+\Cour{x_2,x_1}=\a^*(\d \l x_1,x_2\r)$,
\end{enumerate}
for all sections $x_1,x_2,x_3\in\Gamma(\A)$. Here $\a^*$ is the dual
map $\a^*\colon T^*M\to \A^*\cong\A$, using the isomorphism given by
the inner product.
\end{definition}

Axioms c1) and c2) say that $\Cour{x,\cdot}$ is a derivation of both the
Courant bracket and of the inner product, while c3) relates the
symmetric part of the bracket with the inner product.  From the
axioms, one can derive the following properties of Courant algebroids:
\begin{enumerate}
\item[p1)] $\Cour{x_1,fx_2}=f\Cour{x_1,x_2}+(\a(x_1)f)x_2,\ $
\item[p2)] $\a(\Cour{x_1,x_2})=[\a(x_1),\a(x_2)],$
\item[p3)] $\a\circ\a^*=0$,
\item[p4)] $\Cour{x,\a^*(\mu)}=\a^*(\L_{\a(x)}\mu)$, 
\item[p5)] $\Cour{\a^*(\mu),x}=-\a^*(\iota_{\a(x)}\d\mu)$
\end{enumerate}
for $f\in C^\infty(M),\ \mu\in \Om^1(M)$, and with $\L_v,\iota_v$
denoting Lie derivative, contraction relative to a vector field $v$.
The properties p1) and p2) are sometimes included as part of the
axioms; their redundancy was first observed by Uchino \cite{uch:rem}. 
The basic examples are as follows: 

\begin{examples}
\begin{enumerate}
\item
A Courant algebroid over $M=\pt$ is just a \emph{quadratic Lie algebra},
i.e. a Lie algebra $\g$ with an invariant non-degenerate symmetric
bilinear form. 
\item \label{def:standard}
The \emph{standard Courant algebroid} $\TM$ over a
  manifold $M$ is the vector bundle $\TM=TM\oplus T^*M$, with
  $\a\colon \TM\to TM$ the
projection along $T^*M$, and with inner product and Courant bracket
\[ \begin{split}
\l (v_1,\mu_1),(v_2,\mu_2)\r&=\iota_{v_2}\mu_1+\iota_{v_1}\mu_2,\\ 
\Cour{(v_1,\mu_1),(v_2,\mu_2)}&
=([v_1,v_2],\,\L_{v_1}\mu_2-\iota_{v_2}\d\mu_1),\end{split}\]
for vector fields $v_i\in \mf{X}(M)$ and 1-forms $\mu_i\in\Om^1(M)$.
More generally, given a closed 3-form $\eta\in\Om^3(M)$ one defines a
Courant algebroid $\TM^{(\eta)}$, by adding an extra term
$(0,\iota_{v_2}\iota_{v_1}\eta)$ to the Courant bracket. 
By a result of \u{S}evera, a Courant algebroid $\A\to M$ is isomorphic to 
$TM^{(\eta)}$ if and only if it is \emph{exact}, in the sense that 
the sequence
\[ 0\to T^*M\xra{\a^*} \A \xra{\a} TM\to 0\]
is exact. 
\end{enumerate}
\end{examples}

\subsection{Dirac structures}\label{subsec:dirac}
A subbundle $E\subset \A$ of a Courant algebroid is called a
\emph{Dirac structure} if it is \emph{Lagrangian} (i.e. $E^\perp=E$)
and if its space $\Gamma(E)$ of sections is closed under the Courant
bracket. This then implies that $E$, with the restriction of the
Courant bracket and the anchor map, is a Lie algebroid. 
For any Dirac structure, the generalized
distribution $\a(E)\subset TM$ is integrable in the sense of
Stefan-Sussmann, hence it defines a generalized foliation of $M$.

The lack of integrability of a given Lagrangian subbundle $E\subset
\A$ is measured by the \emph{Courant tensor} $\Upsilon^E\in
\Gamma(\wedge^3 E^*)$,
\[ \Upsilon^E(x_1,x_2,x_3)=
\l x_1, \Cour{x_2,x_3}\r,\ \ x_i\in\Gamma(E) .\]
(Using the Courant axioms, one checks that the right hand side is
tensorial and anti-symmetric in $x_1,x_2,x_3\in \Gamma(E)$.)

\begin{example}
  Let $\pi \in \mf{X}^2(M)=\Gamma(\wedge^2TM)$ be a bivector on
  $M$. Then the graph of $\pi$,
\[ \on{Gr}_\pi=\{(\iota_\mu\pi,\mu)|\ \mu\in T^*M\}\subset \TM\]
is a Lagrangian subbundle. Conversely, a Lagrangian subbundle of
$\TM$ is of the form $\on{Gr}_\pi$ if and only if its intersection
with $TM$ is trivial. Using the pairing of $TM$ and $\on{Gr}_\pi$
given by the inner product, we may identify $\on{Gr}_\pi^*\cong TM$
and hence view the Courant tensor of $\on{Gr}_\pi$ as a section of
$\wedge^3 TM$.  A direct calculation using the definition of
the Courant bracket on $\TM$ shows \cite{cou:di}
\begin{equation}\label{eq:courpo}
 \Upsilon^{\on{Gr}_\pi}=\hh [\pi,\pi]\end{equation}
where $[\pi,\pi]$ is the Schouten bracket of multi-vector fields. Thus
$\on{Gr}_\pi$ is a Dirac structure if and only if $\pi$ is a Poisson
bivector. 
\end{example}

\subsection{Courant morphisms}
Suppose $\A,\A'$ are Courant algebroids
over $M,M'$. We denote by $\ol{\A}$ the Courant algebroid $\A$ with 
the opposite inner product. Given a smooth map 
$\Phi\colon M\to M'$ let 
\[ \on{Graph}(\Phi)=\{(\Phi(m),m)|m\in M\}\subset M'\times M\]
be its graph. The notion of a Courant morphism is due to {\v{S}}evera
\cite{sev:so} (in terms of NQ manifolds) and Alekseev-Xu
\cite{al:der}, see also \cite{bur:cou,pon:ham}.
\begin{definition}
  Let $\A,\A'$ be Courant algebroids over $M,M'$. A \emph{Courant
    morphism} $R_\Phi\colon \A\da \A'$ is a smooth map $\Phi\colon
  M\to M'$ together with a subbundle
\[ R_\Phi\subset (\A'\times \bar{\A})|_{\on{Graph}(\Phi)}\]
with the following properties: 
\begin{enumerate}
\item[m1)] $R_\Phi$ is Lagrangian (i.e. $R_\Phi^\perp=R_\Phi$). 
\item[m2)] The image $(\a'\times\a)(R_\Phi)$ is tangent to the graph
      of $\Phi$. 
\item[m3)] If $x_1,x_2\in\Gamma(\A'\times\ol{\A})$ restrict to
      sections of $R_\Phi$, then so does their Courant bracket
      $\Cour{x_1,x_2}$.
\end{enumerate}
\end{definition}
The composition of Courant morphisms is given as the fiberwise
composition of relations: $R_{\Psi\circ\Phi}=R_\Psi\circ R_\Phi$. (One
imposes the usual transversality conditions to ensure that this
composition is a smooth subbundle.)  For a Courant morphism
$R_\Phi\colon \A\da \A'$, and elements $x\in\A,\ x'\in \A'$, we write
$x\sim_{R_\Phi} x'$ if $(x',x)\in R_\phi$.  Similarly, for
$x\in\Gamma(\A),x'\in\Gamma(\A')$ we write $x\sim_{R_\Phi} x'$ if
$x'\times x$ restricts to a section of $R_\Phi$.  The definition of a
Courant morphism shows that for all sections $x_i\in \Gamma(\A),\ x_i'\in\Gamma(\A')$
\[ x_1\sim_{R_\Phi} x_1',\ x_2\sim_{R_\Phi} x_2'\ \ \Rightarrow\ \ \Cour{x_1,x_2}\sim_{R_\Phi}
\Cour{x_1',x_2'},\ \ \l
x_1,x_2\r=\Phi^* \l x_1',x_2'\r.\]
From now on, we will often describe the subbundle $R_\Phi$ in terms
of the relation $\sim_{R_\Phi}$.

\begin{examples}\label{ex:examples}
\begin{enumerate}
\item An automorphism of a Courant algebroid is an invertible morphism
      $R_\Phi\colon \A\da \A$,  
with $\Phi=\on{id}_M$ the identity map on $M$.
\item
A morphism $\A\da 0$ to the zero Courant algebroid over a point $\pt$
is the same thing as a \emph{Dirac structure} in $\A$.   
\item
Suppose $\Phi\colon M\to M'$ is a given map, $\A,\A'$ are Courant
algebroids over $M,M'$, and $R\subset \A'\times\bar{\A}$ is a Dirac
structure whose foliation is tangent to the graph of $\Phi$. Then the
restriction $R_\Phi=R|_{\on{Graph}(\Phi)}$ is a Courant morphism.
However, not every Courant morphism arises in this way: For instance,
the bracket on $\Gamma(R_\Phi)$ need not be skew-symmetric. 
\item
Associated to any smooth map $\Phi\colon M\to M'$ is a \emph{standard
morphism} $R_\Phi\colon \TM\da \TM'$, where 
\[ (v,\mu)\sim_{R_\Phi} (v',\mu')\ \Leftrightarrow \ v'=\Phi_*v,\ \mu=\Phi^*\mu'.\]
%
%
\item 
If $\g,\g'$ are quadratic Lie algebras (viewed as Courant algebroids
over a point), a Courant morphism $R\colon \g\da \g'$ is a
Lagrangian subalgebra of $\g'\oplus \ol{\g}$.  
\end{enumerate}
\end{examples}
Another example is described in the
following Proposition.
\begin{proposition}[Diagonal morphism]\label{prop:diagonal}
Let $\A\to M$ be any Courant algebroid, and $\diag\colon M\to M\times
M$ the diagonal embedding.  There is a  Courant morphism
\[ R_{\on{diag}}\colon \TM\da \A\times\ol{\A}\] 
given by 
\[ (v,\mu)\sim_{R_\diag} (x,y)\ \Leftrightarrow\ v=\a(x),\ x-y=\a^*\mu.\]
%
%
\end{proposition}
\begin{proof}
  Write $f(x,\mu)=(x,x-\a^*(\mu),\a(x),\mu)$, so that $R_\diag$ is
  spanned by the elements $f(x,\mu)$ with $(x,\mu)\in \A\oplus T^*M$.
  The subbundle $R_{\on{diag}}$ is Lagrangian since
\[ \l f(x,\mu),\ f(x,\mu)\r=\l x,x\r-\l x-\a^*(\mu),x-\a^*(\mu)\r-2\iota_{\a(x)}\mu=0 .\]
Using p3), we see that the image
of $f(x,\mu)$ under the anchor map for $\A\times\ol{\A\times \TM}$ is 
$(\a(x),\a(x),\a(x))$ which is tangent to the graph of $\diag$. 
Now let $x_1,x_2\in \Gamma(\A)$ and $\mu_1,\mu_2\in
\Gamma(T^*M)$. Then $f(x_i,\mu_i)$ are sections of 
$\A\times\ol{\A \times\TM}$ that restrict to sections of
$R_\diag$. Using p4), p5) we obtain 
\[ \Cour{f(x_1,\mu_1),\ f(x_2,\mu_2)}=f(x,\mu)\]
with $x=\Cour{x_1,x_2},\ \ 
\mu=\L_{\a(x_1)}\mu_2-\iota_{\a(x_2)}\d\mu_1$. Hence, the Courant
bracket again restricts to a section of $R_\diag$.  This completes the
proof (cf. \cite[Remark 2.5]{bur:cou}).
\end{proof}

\section{Constructions with Courant algebroids}\label{sec:con}
In this section we describe three constructions involving Courant
algebroids. We begin by discussing a reduction procedure relative to
co-isotropic subbundles. We then use reduction to define restrictions
and pull-backs $\Phi^!\A$ of Courant algebroids. Finally, we introduce the
notion of a Courant algebroid action, and describe conditions under
which the usual pull-back as a vector bundle $\Phi^*\A$ acquires the
structure of a Courant algebroid.
\subsection{Coisotropic reduction}
The reduction procedure is frequently used to obtain new Courant
algebroids out of old ones. See \cite{bur:red} and \cite{zam:red} for
similar constructions for the case of exact Courant algebroids. 
\begin{proposition}\label{prop:cois}
Let $S\subset M$ be a submanifold, and $C\subset \A|_S$ a subbundle
such that 
\begin{enumerate}
\item[r1)] $C$ is co-isotropic (i.e. $C^\perp\subset C$), 
\item[r2)] $\a(C)\subset TS,\ \a(C^\perp)=0$,  
\item[r3)] if $x,y\in \Gamma(\A)$ restrict to sections of $C$, then so does their Courant bracket. 
\end{enumerate}
Then the anchor map, bracket and inner product on $C$ descend to
$\A_C=C/C^\perp$, and make $\A_C$ into a Courant algebroid over $S$.
The inclusion $\Phi\colon S\hra M$ lifts to a  
Courant   morphism 
\[ R_\Phi\colon \A_C\da \A,\ \  y\sim_{R_\Phi} x\ \Leftrightarrow\ x\in C,\ y=q(x)\]
where $q\colon C\to \A_C$ is the quotient map. 
\end{proposition}
\begin{proof}
  By r1),r2) the inner product and the anchor map descend to $\A_C$.
We want to define the bracket on $\A_C$ by 
\[ \Cour{q(x|_S),q(y|_S)}=q(\Cour{x,y}|_S),\]
for any sections $x,y\in\Gamma(\A)$ that restrict to sections of
$C$. To see that this is well-defined, we must show that the 
right hand side vanishes if $q(x|_S)=0$ or if $q(y|_S)=0$. 
Equivalently, letting  $x,y,z\in \Gamma(\A)$ with
  $x|_S,\ y|_S,\ z|_S\in\Gamma(C)$, we must show that $ \l
  \Cour{x,y},z\r$ vanishes on $S$ if one of the two sections $x|_S,\ y|_S$
  takes values in $C^\perp$. (i) Suppose $y|_S\in \Gamma(C^\perp)$. Then
\[ \l \Cour{x,y},z\r=\a(x)\l y,z\r-\l y,\Cour{x,z}\r\]
vanishes on $S$, since $\l y,z\r|_S=0$ and $\a(x)$ is
tangent to $S$. (ii) Suppose $x|_S\in \Gamma(C^\perp)$. Then 
\[ \l \Cour{x,y},z\r = -\l \Cour{y,x},z\r +\l \d\l x,y\r,\a(z)\r\]
vanishes on $S$: The first term vanishes by (i), and the second term
vanishes since $\l x,y\r|_S=0$ and consequently $\d\l x,y\r|_S\in
(\Gamma(\on{ann}(TS))$. The Courant axioms for $\A_C$ follow from the
Courant axioms for $\A$.
\end{proof}

\begin{example}
  Suppose $\ker(\a)=C$ is a smooth subbundle of $\A$, as happens for
  example if the anchor map $\a$ is surjective. Then
  $C^\perp=\ran(\a^*)$, and $\A_C$ is a Courant algebroid with trivial
  anchor map. That is, $\A_C$ is simply a bundle of quadratic Lie
  algebras.
\end{example}

\begin{remark}
%
%
  Assume $S=M$, and let $C\subset\A$ be as in the Proposition. Write
  $R=R_\Phi$ where $\Phi=\on{id}_M$. Let $R^t\colon \A\da \A_C$ be the
  transpose of $R$, i.e.  $(y,x)\in R^t \Leftrightarrow (x,y)\in R$.
  Then $R^t\circ R$ is the identity morphism of $\A_C$, while the
  composition $T=R\circ R^t\colon \A\da \A$ satisfies $T\circ T=T$.
  Explicitly,
\[ x\sim_T x'\ \Leftrightarrow\ x,x'\in C,\ q(x)=q(x').\]
This is parallel to a construction in symplectic geometry, see 
Guillemin-Sternberg \cite{gui:lin}.  
\end{remark}

\subsection{Pull-backs of  Courant algebroids}\label{subsec:pull}
Let $\A\to M$ be a given Courant algebroid, with anchor map $\a\colon
\A\to TM$.
\begin{proposition}\label{prop:restrict}
Suppose $S\subset M$ is an embedded submanifold, and assume that 
 $\a$ is
transverse to $TS$, in the sense that
\begin{equation}\label{eq:transvert}
 TM|_S=TS+\ran(\a)|_S.\end{equation}
Then 
\begin{equation}
 \A_S:=\a^{-1}(TS)/\a^*(\on{ann}(TS\cap
\on{ran}(\a)))\end{equation}
is a Courant algebroid $\A_S\to S$ called the \emph{restriction of $\A$ 
to $S$}. One has $\rk(\A_S)=\rk(\A)-2\dim
M+2\dim S$. The inclusion $\Phi\colon S\hra M$ lifts to a Courant
morphism $\A_S\da \A$.
\end{proposition}
Of course, the restriction $\A_S$ as a Courant algebroid is different
from the restriction $\A|_S$ as a vector bundle.
\begin{proof}
  The transversality condition ensures that $C=\a^{-1}(TS)$ is a smooth
  subbundle. We
  verify the conditions from Proposition \ref{prop:cois}: r1) and r2)
  follow since $\ker(\a)\subset C$, hence
\[ C^\perp\subset
\ker(\a)^\perp=\ran(\a^*)\subset \ker(\a)\subset C.\] 
Condition r3) follows from p2), and since the Lie bracket of two
vector fields tangent to $S$ is again tangent to $S$.  To describe
$C^\perp$, note that $\a^*\mu$ is orthogonal to $C$ if and only if
$\mu$ annihilates all elements of $\a(C)=TS\cap \ran(\a)$.  Hence
$C^\perp= \a^*(\on{ann}(TS\cap \on{ran}(\a)))$. Since $C$ is the
kernel of a surjective map $\A|_S\oplus TS\to TM|_S$, one obtains
$\rk(C)=\rk(\A)-\dim M+\dim S$. Consequently $\rk(C^\perp)=\dim M-\dim
S$, and the formula for $\rk(\A_C)$ follows.
\end{proof}

\begin{remark}
  The transversality condition \eqref{eq:transvert} may be replaced by
  the weaker assumption that $C$ is a smooth subbundle. Note also that
  the transversality condition is automatic if $\ran(\a)=TM$, e.g. for
  exact Courant algebroids.  For this case, restriction of Courant
  algebroids is discussed in \cite[Lemma 3.7]{bur:red} and in
  \cite[Appendix]{gua:ge2}.
\end{remark}

Restriction to submanifolds generalizes to pull-back under maps: 
\begin{definition}
Suppose $\Phi\colon S\to M$ is a smooth map whose
differential $\d\Phi\colon TS\to TM$ is transverse to $\a\colon \A\to
TM$.  We define a \emph{pull-back Courant algebroid} $\Phi^!\A\to S$ by
restricting the direct product $\A\times \TS \to M\times S$ to the
graph $\on{Graph}(\Phi)\cong S$. (The shriek notation is used to
distinguish $\Phi^!A$ from the pull-back as a vector bundle.)
\end{definition}
Explicitly, we have the following description: 

\begin{proposition}
The pull-back Courant algebroid is a quotient $\Phi^!\A=C/C^\perp$ 
where 
\[\begin{split} C&=\{(x;v,\mu)\in \A\times \T S,\ (\d\Phi)(v)=\a(x)
  \}.\\ 
C^\perp&=\{(\a^*\lambda;0,\nu)|\  
(\d\Phi)^*\lambda+\nu\in \on{ann}((\d\Phi)^{-1}\ran(\a))\}.\end{split}\]
One has $\rk(\Phi^!\A)=\rk(\A)-2(\dim M-\dim S)$. 
\end{proposition}
\begin{proof}
This is just a special case of Proposition \ref{prop:restrict}. Let us
nevertheless give some details on the computation of $C^\perp$. 
As in the proof of Proposition \ref{prop:restrict}, $C^\perp$ is 
contained the range of the dual of the anchor map. 
Hence, elements of $C^\perp$ are of the form $(\a^*\lambda,(0,\nu))$
with $\lambda\in T^*M$ and $\nu\in T^*S$. Pairing with $(x;v,\mu)\in
C$ we obtain the condition 
\[\begin{split}
0&=\l (\a^*\lambda;0,\nu),\ (x;v,\mu)\r\\
&=\l\lambda,\a(x)\r+\l\nu,v\r\\
&=\l\lambda,\d\Phi(v)\r+\l\nu,v\r\\
&=\l (\d\Phi)^*\lambda+\nu,v\r.\end{split}\]
Thus, $(\a^*\lambda;0,\nu)\in C^\perp$ if and only if
$(\d\Phi)^*\lambda+\nu\in T^*S$ annihilates all $v\in TS$ with 
$\d\Phi(v)\in \ran(\a)$. 
\end{proof}
Note that for $\ran(\a)=TM$ (e.g. exact Courant algebroids) the
description of $C^\perp$ simplifies to the condition
$(\d\Phi)^*\lambda+\nu=0$.

\begin{proposition}
  If $\Phi\colon S\to M$ is an embedding, with $\a$ transverse to
  $TS$, then the pull-back $\Phi^!\A$ is canonically isomorphic to the
  restriction of $\A$ to $S$.
\end{proposition}
\begin{proof}
Recall that $\A_S=\a^{-1}(TS)/\a^*(\on{ann}(TS\cap
\on{ran}(\a)))$. The inclusion 
\[ \psi\colon \a^{-1}(TS) \to (\A\times \T S)|_{\on{Graph}(\Phi)},\ \ x\mapsto (x;\a(x),0),\]
takes values in $C$, with $ \psi(\a^{-1}(TS))\cap C^\perp=\psi(\a^*(\on{ann}(TS\cap
\on{ran}(\a))))$. 
The resulting inclusion $\A_S\to \Phi^!\A$ is an isomorphism of vector
bundles by dimension count, and it clearly preserves inner products. It is an isomorphism of Courant algebroids
since for all sections $x_1,x_2\in \Gamma(\A)$, such that $x_1|_S,\ 
x_2|_S$ take values in $\A_S$, the Courant bracket $x=\Cour{x_1,x_2}$
satisfies
\[\begin{split}
\Cour{(x_1;\a(x_1),0),(x_2;\a(x_2),0)}|_{\on{Graph}(\Phi)}
&=(\Cour{x_1,x_2};[\a(x_1),\a(x_2)],0)|_{\on{Graph}(\Phi)}\\
&=(\Cour{x_1,x_2};\a(\Cour{x_1,x_2}),0)|_{\on{Graph}(\Phi)}\\
&=\psi(\Cour{x_1,x_2}|_S)
.\qedhere\end{split}\]
\end{proof}

\begin{proposition}
For any smooth map $\Phi\colon S\to M$, one has a canonical
isomorphism 
\[ \Phi^!(\TM)=\T S.\]
\end{proposition}
\begin{proof}
If $\A=\T M$ the bundle $C\subset (\T M\times \T S)|_{\on{Graph}(\Phi)}$ has the description, 
\[ C=\{(\d\Phi(v),\lambda;\,v,\mu)|\ \lambda\in T^*M,\  v\in TS\}.\]
The bundle $\T S$ is embedded in $C$ as the subbundle defined by $\lambda=0$,
and it defines a complement to the subbundle $C^\perp$ given by the
conditions $v=0,\ \mu=-(\d\Phi)^*\lambda$. Since the inclusion $\T
S\to C$ preserves Courant brackets, this shows $C/C^\perp=\T S$ as
Courant algebroids.
\end{proof}

\begin{proposition}\label{prop:pullbacks}
  There is a canonical Courant morphism
\begin{equation}\label{eq:pphi} P_\Phi\colon \Phi^!\A\da \A\end{equation}
lifting $\Phi\colon S\to M$. Explicitly, 
\[ y\sim_{P_\Phi} x\ \Leftrightarrow\ \exists v\in TS\colon
\a(x)=\d\Phi(v),\ (x;v,0)\in C \mbox{ maps to }y.\]
\end{proposition}
Note that as a space, $P_\Phi$ is the fibered product of $\A$ and $TS$
over $TM$. It is a smooth vector bundle over $S\cong \on{Graph}(\Phi)$
since $\a$ is transverse
to $\d\Phi$ by assumption.
\begin{proof}
  We factor $\Phi=\Phi_1\circ \Phi_2$, where $\Phi_2\colon S\to
  M\times S$ is the inclusion as $\on{Graph}(\Phi)$, and $\Phi_1\colon
  M\times S\to M$ is projection to the first factor.  Since
  $\Phi^!\A=(\A\times \T S)|_{\on{Graph}(\Phi)}$, there is a canonical
  Courant morphism $R_{\Phi_2}\colon \Phi^!\A\da \A\times \T S$ as
  explained in Proposition \ref{prop:restrict}. The
  Lagrangian subbundle $R_{\Phi_2}$ consists of elements of the form
  $(x;v,\mu;y)$ with $(x;v,\mu)\in C$ mapping to $y\in C/C^\perp$.  On
  the other hand, the projection map $\Phi_1\colon M\times S\to M$
  lifts to a Courant morphism $R_{\Phi_1}\colon \A\times\T S\da \A$
  (given as the direct product of the identity morphism $\A\da \A$
  with the standard morphism $\T S\to \T (\pt)=\pt$). Thus
  $R_{\Phi_1}$ is given by elements of the form $(x;x;v,0)$ with $x\in
  \A$ and $v\in TS$.  Composition gives the Courant morphism
  $P_\Phi\colon \Phi^!\A\da \A$ as described above.
\end{proof}

Our construction of a pull-back Courant algebroid is similar to the
notion of \emph{pull-back Lie algebroid}, due to Higgins-Mackenzie \cite{hig:alg}. Suppose $E\to
M$ is a Lie algebroid with anchor map $\a\colon E\to TM$, and
$\Phi\colon S\to M$ is a smooth map such that $\a$ is transverse to
$\d\Phi$. One defines
\begin{equation}\label{eq:pullalg}
 \Phi^!E=\{(x,v)|\ \d\Phi(v)=\a(x)\}\subset (E\times
TS)|_{\on{Graph}(\Phi)}.\end{equation}
By the transversality assumption, this is a subbundle of rank
$\rk(\Phi^!E)=\rk(E)-\dim(M)+\dim(S)$.  Given two sections of $E\times
TS$ whose restriction to $\on{Graph}(\Phi)$ takes values in $\Phi^!E$,
then so does their Lie bracket. Furthermore, if one of the two
sections vanishes along $\on{Graph}(\Phi)$, then so does their Lie
bracket.  This defines a Lie bracket on $\Gamma(\Phi^!E)$, making
$\Phi^!E$ into a Lie algebroid. If $E$ is a Lagrangian sub-bundle of
$\A\to M$, and such that $\d\Phi$ is transverse to both $\a,\ \a|_E$,
then $\Phi^!E$ is embedded as a Lagrangian sub-bundle of $\Phi^!\A$ by
the map $\Phi^!E\to \Phi^!\A$, taking $(x,v)\in E\times TS$ with
$\a(x)=\d\Phi(v)$ to the equivalence class $[(x;v,0)]\in \Phi^!\A$ of
$(x;v,0)\in C$. Clearly, if $E$ is a Dirac structure, then the Courant
bracket on $\Phi^!\A$ restricts to the Lie algebroid bracket on
$\Phi^!E$, and in particular $\Phi^!E$ is a Dirac structure in
$\Phi^!\A$.

\subsection{Actions of Courant algebroids}\label{subsec:action}
The pull-back $\Phi^!\A$ of a Courant algebroid is different from the
pull-back as a vector bundle $\Phi^*\A$. In order to define a Courant algebroid
structure on $\Phi^*\A$, one needs the additional structure of an
action of $\A$ on the map $\Phi$, such that all stabilizers of the
action are co-isotropic.  Here actions of Courant algebroids may be
defined in analogy with the actions of Lie algebroids \cite{hig:alg}:
\begin{definition}
Suppose $\A\to M$ is a Courant algebroid. An \emph{action of $\A$ on a
manifold $S$} is a map $\Phi\colon S\to M$ together with an
\emph{action map} $\varrho\colon \Gamma(\A)\to
\mf{X}(S)$ satisfying
\[ \begin{split}
\d\Phi\circ \varrho(x)&=\a(x),\\ 
[\varrho(x),\varrho(y)]&=\varrho(\Cour{x,y}),\\ 
\varrho(fx)&=\Phi^*f \ \varrho(x).\end{split}\]
for all $x,y\in\Gamma(\A)$ and $f\in C^\infty(M)$.
\end{definition}
The last condition shows that $\varrho$ defines a vector bundle map
(denoted by the same letter)
\begin{equation}\label{eq:varrho} \varrho\colon \Phi^*\A\to TS.\end{equation}
For each $s\in S$, the kernel of the map $\varrho_s\colon
\A_{\Phi(s)}\to T_sS$ is called the \emph{stabilizer at $s$}.

\begin{theorem}
  Suppose the Courant algebroid $\A\to M$ acts on $S$ with
  co-isotropic stabilizers. Let $\Phi\colon S\to M$ and $\varrho\colon
  \A\to TS$ be the maps defining the action. Then the pull-back vector
  bundle $\Phi^*\A$ carries a unique structure of Courant algebroid
  with anchor map $\varrho$, such that the pull-back map on sections
  $\Phi^*\colon \Gamma(\A)\to \Gamma(\Phi^*\A)$ preserves inner
  products and Courant brackets. If $E\subset \A$ is a Dirac
  structure, then $\varrho|_E$ is an action of the Lie algebroid $E$
  on $S$, and the resulting Lie algebroid structure on $\Phi^*E$
  coincides with that as a Dirac structure in $\Phi^*\A$.
\end{theorem}
\begin{proof}
  Consider the direct product Courant algebroid $\A\times \T S$ over
  $M\times S$. For $\mu\in \Om^1(S)$ and $x\in \Gamma(\A)$, let 
\[ f(x,\mu)=(x;\varrho(x),\mu)\in \Gamma(\A\times \T S).\]
The image of $f(x,\mu)$ under the anchor map is the vector field
$(\a(x),\varrho(x))$, which is tangent to the graph of $\Phi$ since
$\d\Phi(\varrho(x))=\a(x)$. The Courant bracket between two sections
of this form is $\Cour{f(x_1,\mu_1),f(x_2,\mu_2)}=f(x,\mu)$ with 
$x=\Cour{x_1,x_2}$ and $\mu=\L_{\varrho(x_1)}\mu_2-\iota_{\varrho(x_2)}\d\mu_1$.
Let $C\subset (\A\times \T S)|_{\on{Graph}(\Phi)}$ be the subbundle
spanned by the restrictions of $f(x,\mu)$ to the graph of $\Phi$. 
This is a co-isotropic subbundle, since its fiber at $(\Phi(s),s)$ 
contains the co-isotropic subspace $\ker(\rho_s)\times T^*_sS$.
Its orthogonal is given by 
\[ C^\perp=\{(-\varrho^*\nu;0,\nu)|\ \nu\in T^*S\}.\]
(Indeed, it is easy to check that the elements on the right hand side
lie in $C^\perp$. Equality follows since $\rk(C)=\rk(\A)+\dim S$,
hence $\rk(C^\perp)=\rk(\A\times\T S)-\rk(C)=\dim S$.) Since the
co-isotropic subbundle satisfies all the conditions from Proposition
\ref{prop:cois}, the reduced Courant algebroid $\A_C$ over
$\on{Graph}(\Phi)\cong S$ is defined. The sections of the form
$f(x,0)$ span a complement to $C^\perp$ in $C$, and identify
$\A_C=\Phi^*\A$ as a vector bundle. This identification also preserves
inner products, since $\l f(x_1,0),f(x_2,0)\r=\l x_1,x_2\r$.
Furthermore, $\Cour{f(x_1,0),f(x_2,0)}=f(\Cour{x_1,x_2},0)$ shows that
the pull-back map on sections $\Gamma(\A)\to
\Gamma(\Phi^*\A)=\Gamma(\A_C)$ preserves Courant brackets. If
$E\subset \A$ is a Dirac structure, then it is obvious that the action
$\varrho$ restricts to a Lie algebroid action of $E$ (i.e. the map
$\Gamma(E)\to \Gamma(TS)$ defined by $\varrho$ preserves brackets).
Also, since the Lie algebroid bracket on $\Gamma(\Phi^*E)$ is determined by 
the Lie bracket on the subspace $\Phi^*\Gamma(E)$, it is immediate that
$\Phi^*E$ with this bracket is a Dirac structure in $\Phi^*\A$. 
\end{proof}

\begin{example}
  Suppose $\A\to M$ is a Courant algebroid, and $\Phi\colon S\hra M$
  is an embedded submanifold, such that $\ran(\a)$ is tangent to $S$.
  Then the map $\Gamma(\A)\to \Gamma(TS)$, given by restriction to $S$
  followed by $\a$, satisfies the axioms. Hence, $\Phi^*\A$ is a
  well-defined Courant algebroid.
\end{example}

\begin{example}\label{ex:important}
  Let $\g$ be a quadratic Lie algebra, viewed as a Courant algebroid
  over a point. Let $M$ be a manifold with a $\g$-action whose
  stabilizer algebras are coisotropic. Then the product $M\times\g$
  (viewed as the vector bundle pull-back of $\g\to \pt$ by the map
  $M\to \pt$) acquires the structure of a Courant algebroid. This
  example will be explored in detail in Section \ref{sec:cou}.
\end{example}

Since the Courant bracket on $\Phi^*\A$ was defined by reduction,
there is a Courant morphism $\Phi^*\A\da \A\times \T S$ lifting the
inclusion $S\hra M\times S$ as the graph of $\Phi$. Its composition
with $\A\times \T S\da \A$ (lifting the projection to the first
factor) is a Courant morphism 
\begin{equation}\label{eq:qphi}
 U_\Phi\colon \Phi^*\A\da \A
\end{equation} 
lifting $\Phi$.

\section{Manin pairs and Manin triples}\label{sec:man}
A pair $(\A,E)$ of a Courant algebroid $\A\to M$ together with a
Dirac structure is called a \emph{Manin pair} (over $M$). Given a
second Dirac structure $F\subset \A$ such that $\A=E\oplus F$, the triple
$(\A,E,F)$ is called a \emph{Manin triple} (over $M$). For $M=\pt$, this reduces to the
classical notion of a Manin triple $(\dd,\g_1,\g_2)$ as introduced by Drinfeld
\cite{dr:qu}: A split quadratic Lie algebra $\dd$ with 
two transverse Lagrangian subalgebras $\g_1,\g_2$. 
\begin{remark}\label{rem:lieb}
  As shown by Liu-Weinstein-Xu \cite{liu:ma}, Manin triples over $M$ are equivalent
  to \emph{Lie bialgebroids} $A\to M$. Indeed, for any Lie bialgebroid
  the direct sum $\A=A\oplus A^*$ carries a unique structure of a
  Courant algebroid such that $(\A,A,A^*)$ is a Manin triple. 
  The more general concept of a \emph{Manin quasi-triple} requires
  only that $E$ is integrable. It is equivalent to the notion of
  quasi-Lie bialgebroid, see Kosmann-Schwarzbach \cite{kos:jac,kos:qua},
  Roytenberg \cite{roy:qua}, and Ponte-Xu \cite{pon:ham}.
\end{remark}

\subsection{Backward images}
The following result was obtained in \cite[Proposition 2.10]{al:pur} 
for the case of exact Courant algebroids, with a very different
proof. 
\begin{proposition}\label{prop:natural}
Let $R_\Phi\colon \A\da \A'$ be a Courant morphism. Suppose $E'\subset
\A'$ is a Lagrangian subbundle, with the property  
\begin{equation}\label{eq:technical}
 x'\in E',\ 0\sim_{R_\Phi} x'\in E'\ \Rightarrow\ x'=0.
\end{equation}
Then the \emph{backward image}
\[ E:=E'\circ R_\Phi=\{x\in\A|\ \exists x'\in E'\colon x\sim_{R_\Phi} x'\}\]
is a Lagrangian subbundle of $\A$, and there is a unique bundle
homomorphism $\alpha\colon E\to \Phi^*E'$ such that $x\sim_{R_\Phi} \alpha(x)$
for all $x\in E$. The Courant tensors of $E,E'$ are related by 
\begin{equation}\label{eq:relation}
 \Upsilon^E=\alpha^*(\Phi^*\Upsilon^{E'}),
\end{equation}
where $\alpha^*$ is the map dual to $\alpha$, extended to the exterior
algebras.
\end{proposition}
\begin{proof}
  The fact that the backward image is a Lagrangian subbundle is
  parallel to Guillemin-Sternberg's construction in symplectic
  geometry \cite{gui:lin}, hence we will be brief. We have
\[ E=E'\circ R_\Phi\cong ((E'\times R_\Phi)\cap C)\big/((E'\times R_\Phi)\cap C^\perp)\]
where $C\subset (\ol{\A}'\times \A'\times\ol{\A})|_{\diag(M')\times M}$
is the co-isotropic subbundle given as 
\[ C=\{(x',x',x)|\ x\in \A,\ x'\in \A'\}.\]
$C^\perp$ is similarly given as the set of all $(x',x',0)$.  The
condition $x'\in E',\, 0\sim_{R_\Phi}x' \Rightarrow x'=0$ amounts to the
transversality property $(E'\times R_\Phi)\cap C^\perp=0$, ensuring
that $E$ is a smooth subbundle. It is easy to see that $E$ is
isotropic, and hence Lagrangian by dimension count.  The map $\alpha$
associates to each $x\in E_m$ the unique $x'\in E_{\Phi(m)}$ with
$x\sim_{R_\Phi}x'$. To prove \eqref{eq:relation} we must 
show that
\begin{equation}\label{eq:courantdiff}
 \Upsilon^{E'}(x_1',x_2',x_3')-\Upsilon^E(x_1,x_2,x_3)=0\end{equation}
for all $x_i\in E_m,\ x_i'=\alpha(x_i)\in E'_{\Phi(m)}$. 
Think of $E'\times E$ as a
Lagrangian subbundle of $\A'\times\ol{\A}$. Its Courant tensor is 
\[ \Upsilon^{E'\times E}=\Upsilon^{E'}-\Upsilon^E\]
as an element of $(\wedge^3 E')^*+\wedge^3 E^*\subset 
\wedge^3(E'\times E)^*$. Let
$\sig_1,\sig_2,\sig_3$ be three sections of $\A'\times\ol{\A}$ 
restricting to sections of $R_\Phi$. Then 
\[  \Upsilon^{E'\times E}(\sig_1,\sig_2,\sig_3)=\l
\sig_1,\Cour{\sig_2,\sig_3}\r \]
vanishes along the graph of $\Phi$: Indeed $\Cour{\sig_2,\sig_3}$ restricts to a section of
$R_\Phi$, hence its inner product with $\sig_1$ restricts to $0$. 
Choosing the $\sig_i$ such that $\sig_i|_{(\Phi(m),m)}=(x_i',x_i)$ 
this gives \eqref{eq:courantdiff}. 
\end{proof}

As a special case, we see that if $(\A',E')$ is a Manin pair, and $E$
is the backward image of $E'$ under a Courant morphism $R_\Phi\colon
\A\to \A'$ satisfying \eqref{eq:technical}, then $(\A,E)$ is a Manin pair.

\begin{example}\label{ex:pullie}
Let $\A\to M$ be a Courant algebroid and $\Phi\colon S\to M$ a
  smooth map, with $\d\Phi$ transverse to the anchor map $\a$. Let
  $P_\Phi\colon \Phi^!\A\da \A$ be the Courant morphism defined in
  Section \ref{subsec:pull}, and suppose $E\subset \A$ is a
  Lagrangian subbundle, with the property $x\in E,\ 0\sim_{P_\Phi}
  x\Rightarrow x=0$. Then
\[ E\circ P_\Phi=\Phi^! E\]
where $\Phi^!E$ is defined by \eqref{eq:pullalg}.

Similarly, given a Courant algebroid action $(\Phi,\varrho)$ of $\A$
as in Section \ref{subsec:action}, the backward image of a Lagrangian
sub-bundle $E$ under the Courant morphism $U_\Phi$ is just the pull-back
bundle $\Phi^*E$. 
\end{example}

%
%

\subsection{The bivector associated to a Lagrangian splitting}
Let $\A\to M$ be a Courant algebroid. By a \emph{Lagrangian splitting} of
$\A$, we mean a direct sum decomposition
\[ \A=E\oplus F\]
where $E,F$ are Lagrangian subbundles. 
Given a Lagrangian splitting $\A=E\oplus F$, let $\pr_E,\pr_F$ be the
projections onto the two summands. Thus $x=\pr_E(x)+\pr_F(x)$ for all 
$x\in \A$. Note $\l x, x\r=2\l \pr_E(x),\pr_F(x)\r$, so that
$\pr_E(x)$ and $\pr_F(x)$ are orthogonal if $x$ is isotropic. This
applies in particular to elements of the form $x=\a^*(\mu)$.
Define a bi-vector field $\pi\in \mf{X}^2(M)$ by the identity 
\[ \iota_\mu\pi=\a(\pr_F(\a^*(\mu)),\ \ \mu\in T^*M.\]
This is well-defined since 
$\iota_\mu\a(\pr_F(\a^*(\mu))=\l\a^*(\mu),\pr_F(\a^*(\mu)\r=0$.
Since $\a(\a^*(\mu))=0$, the formula for $\pi$ may also be written 
$ \iota_\mu\pi=-\a(\pr_E(\a^*(\mu))$. If $e_i$ is a local frame for 
$E$, and $f^i$ the dual frame for $F$, i.e. 
$\l e_i,f^j\r=\delta_i^j$, we find, 
\begin{equation}\label{eq:piformula}
 \pi=\hh \sum_i \a(e_i)\wedge \a(f^i).
\end{equation}
Note that $\pi$ changes sign if the roles of $E,F$ are reversed.  
\begin{example}\label{ex:pi}
If $\pi$ is any bivector on $M$, with graph $\on{Gr}_\pi\subset \TM$,
the splitting $\TM=TM\oplus \on{Gr}_\pi$ is Lagrangian splitting. The 
bivector associated to this splitting is just $\pi$ itself. 
\end{example}
The
following Proposition gives an alternative description of the bivector
field $\pi$.
\begin{proposition}\label{prop:prel}
Let $\A=E\oplus F$ be a Lagrangian splitting with associated bivector
$\pi$. Then $\on{Gr}_{-\pi}\subset \TM$ 
is the backward image of $E\times F$ under
the diagonal morphism  $R_\diag$ from Proposition \ref{prop:diagonal}. 
Furthermore, $0\sim_{R_\diag}(x,y)\in E\times F$ implies $x=0,\ y=0$.
\end{proposition}
\begin{proof}
By definition of the diagonal morphism, we have 
$(v,\mu)\sim_{R_\diag}(x,y)$ if and only if 
$v=\a(x),\ x-y=\a^*(\mu)$. 
In particular, for $\mu=0$ and $x\in E,\ y\in F$ we obtain $x=y=0$.
Hence, Proposition \ref{prop:natural} shows that the backward image of
$E\times F$ under $R_\diag$ is a smooth Lagrangian subbundle of
$\TM$, transverse to $TM$. The backward image is thus of the form
$\on{Gr}_{-\pi}$ for \emph{some} bivector $\pi$. The relation
\[ (-\iota_\mu\pi,\ \mu)\sim_{R_\diag} (x,y)\]
with $x\in E$ and $y\in F$ means by definition of $R_\diag$ that 
$\a^*\mu=x-y$, thus $y=\pr_F \a^*\mu$, and $-\iota_\mu\pi=\a(y)$. 
Thus $\iota_\mu\pi=
\a(\pr_F(\a^*(\mu)))$ proving the Formula \eqref{eq:piformula}. 
\end{proof}

\subsection{Rank of the bivector}
The rank of the map $\pi^\sharp\colon T^*M\to TM,\ \mu\mapsto
\iota_\mu\pi$ at $m\in M$ is called the rank of $\pi$ at $m$. If  
$\pi$ is integrable (i.e. Poisson), the range of $\pi_m^\sharp$ is the
tangent space to the symplectic leaf, and so $\rk(\pi_m)$ is its
dimension. By definition of $\pi$, 
\begin{equation}\label{eq:range}
 \ran\pi^\sharp=\a(\pr_F \ran\a^*)=\a(\pr_E\ran\a^*).
\end{equation}
For each $m$, the subspace $L_m=\ran(\a^*_m)+(\ker(\a_m)\cap F_m)$ is
Lagrangian, as one verifies by taking its orthogonal. We have:
\begin{lemma}
At any point $m\in M$, the rank of the bivector $\pi$ is given by the
formula, 
\begin{equation}\label{eq:rk}
 \rk(\pi_m)=\dim(\a_m(F_m))-\dim(L_m\cap E_m).\end{equation}
\end{lemma}
\begin{proof}
Formula \eqref{eq:range} shows that $\pi_m$ has rank 
\[ \rk(\pi_m)=\dim(\pr_E(\ran(\a^*_m)))-\dim(\ker(\a_m)\cap
\pr_E(\ran(\a^*_m))).\]
But $\pr_E(\ran(\a^*_m))^\perp=((F_m+\ran(\a^*_m))\cap E_m)^\perp=(F_m\cap
\ker(\a_m))+E_m$ has dimension $\dim E_m+\dim(F_m\cap\ker(\a_m))$. Hence
$\dim(\pr_E(\ran(\a^*_m)))=\dim F_m-\dim(F_m\cap\ker(\a_m))=\dim(\a_m(F_m))$.
Similarly, one observes that 
\[ \ker(\a_m)\cap \pr_E(\ran(\a^*_m))=(\ran(\a^*_m)+F_m\cap
\ker(\a_m))\cap E_m.\] 
\end{proof}
Equation \eqref{eq:range} shows in particular that
$\ran(\pi^\sharp)\subset \a(E)\cap \a(F)$.
In nice cases, this is an equality:
\begin{proposition}\label{prop:leaves}
Let $\A=E\oplus F$ be a Lagrangian splitting, with associated bivector
$\pi\in\mf{X}^2(M)$. Then
\begin{equation}\label{eq:equal} 
\ran(\pi^\sharp)= \a(E)\cap \a(F),\end{equation}
if and only if
\begin{equation}\label{eq:condition}
 \ker(\a)=\on{ran}(\a^*)+(\ker(\a)\cap E)+(\ker(\a)\cap F).
\end{equation}
\end{proposition}
\begin{proof}
Suppose
  Condition \eqref{eq:condition} is satisfied. Given $v\in \a(E)\cap
  \a(F)$, let $x\in E,\ y\in F$ with $\a(x)=\a(y)=v$. Then
  $x-y\in \ker(\a)$, and \eqref{eq:condition} allows us to modify
  $x,y$ to arrange $x-y\in \on{ran}(\a^*)$. We may thus write
  $x-y=\a^*(\mu)$. But $v=\a(x),\ x-y=\a^*(\mu)$ means that
  $(v,\mu)\sim_{R_\diag}(x,y)$.  Thus
  $v=-\iota_\mu\pi\in\ran\pi^\sharp$, which proves \eqref{eq:equal}.  
  Conversely, assume \eqref{eq:equal} and let $w\in
  \ker(\a)$ be given. Write $w=x-y$ with $x\in E,\ y\in F$, and put
  $v=\a(x)=\a(y)$. 
  By \eqref{eq:equal} there exists $\mu\in T^*M$ with
  $v=-\iota_\mu\pi$. By Proposition \ref{prop:prel}, we have
  $(v,\mu)\sim_{R_\diag} (\ti{x},\ti{y})$, for some $\ti{x}\in E$,
  $\ti{y}\in F$. Thus $\ti{x}-\ti{y}=\a^*(\mu)$ and $\a(\ti{x})=\a(\ti{y})=v$. This gives
  the desired decomposition
\[ w=x-y=\a^*(\mu)+(x-\ti{x})-(y-\ti{y})\]
with $x-\ti{x}\in E\cap \ker(\a)$ and $y-\ti{y}\in F\cap \ker(\a)$. 
\end{proof}
 
\begin{remark}
For an exact Courant algebroid, $\ker(\a)=\ran(\a^*)$ and so
Condition \eqref{eq:condition} is automatic. 
\end{remark}

\subsection{Integrability of the bivector}

\begin{theorem}\label{th:splittings}
Let $\A=E\oplus F$ be a Lagrangian splitting, with associated bivector
$\pi\in\mf{X}^2(M)$, and let
\[ \Upsilon^E\in \Gamma(\wedge^3 F),\ \ \Upsilon^F\in \Gamma(\wedge^3
E)\]
be the Courant tensors, where we are using the isomorphisms $E^*\cong F,\ F^*\cong E$
given by the pairing between $E,F$. Then 
\begin{equation}\label{eq:main}
 \hh [\pi,\pi]=\a(\Upsilon^E)+\a(\Upsilon^F).
\end{equation}
In particular, if $(\A,E,F)$ is a Manin triple over $M$, then $\pi$ is a Poisson
structure. 
The symplectic leaves of that Poisson structure are contained in 
the the connected components of the intersections of the leaves of the Dirac structures 
$E$ and $F$. Under Condition \eqref{eq:condition} they are equal to these
components. 
\end{theorem}

\begin{proof}
Since $\on{Gr}_{-\pi}$ is the backward image of $E\times F\subset\A\times\ol{\A}$, Proposition
\ref{prop:natural} shows that its Courant tensor is given by 
\begin{equation}\label{eq:preliminary}
 \Upsilon^{\on{Gr}_{-\pi}}=\alpha^*(\pr_1^*\Upsilon^E-\pr_2^*\Upsilon^F).
\end{equation}
Here $\pr_1^*\colon \wedge F^*\to \wedge(E\times F)^*$ and
$\pr_2^*\colon \wedge F^*\to \wedge(E\times F)^*$ are pull-backs under
the two projections. (Note that the Courant tensor of $F$, viewed as a
subbundle of $\ol{\A}$, is $-\Upsilon^F$.)
The map $\alpha\colon \on{Gr}_{-\pi}\to E\times F$
was computed in the proof of Proposition \ref{prop:prel}: 
\[ \alpha(-\iota_\mu\pi,\mu)=\big(\pr_E(\a^*(\mu)),-\pr_F(\a^*(\mu))\big).\]
To calculate the dual map
\[ \alpha^*\colon F\oplus E\cong E^*\oplus
F^*\to \on{Gr}_{-\pi}^*\cong TM,\] 
let $x\in E, y\in F$ and $\mu\in T^*M$. We have
\[ \begin{split}
\l \alpha^*(y,x),(-\iota_\mu\pi,\mu)\r&=\l (y,x),\
\alpha((-\iota_\mu\pi,\mu))\r\\
&=\l y,\pr_E(\a^*(\mu))\r-\l x,\pr_F(\a^*(\mu))\r\\
&=\l\mu,\a(y)-\a(x)\r.
\end{split}\]
Hence, $\alpha^*(y,x)=\a(y)-\a(x),\ \ x\in E, y\in F$. It follows that
\[ 
\alpha^*(\pr_1^*\Upsilon^E-\pr_2^*\Upsilon^F)=
\a(\Upsilon^E)+\a(\Upsilon^F).\]
On the other hand, as remarked in Section \ref{subsec:dirac} we also
have $\Upsilon^{\on{Gr}_{-\pi}}=\hh [\pi,\pi] $ using the
identification $\on{Gr}_{-\pi}^*\cong TM$. Hence
\eqref{eq:preliminary} translates into \eqref{eq:main}.  If both $E,F$
are Dirac structures the tensors $\Upsilon^E,\Upsilon^F$ vanish, and
\eqref{eq:main} shows that $[\pi,\pi]=0$. The description of the
symplectic leaves follows from Proposition \ref{prop:leaves}.
\end{proof} 

The fact that pairs of transverse Dirac structures define Poisson
structures is due to Mackenzie-Xu \cite{mac:lie}. Their paper
expresses this result in terms of Lie bialgebroids (cf. Remark
\ref{rem:lieb}).

\subsection{Relations of Lagrangian splittings}
Suppose $\A=E\oplus F$ and $\A'=E'\oplus F'$ are Lagrangian splittings
of Courant algebroids over $M,M'$, with associated bivectors
$\pi,\pi'$. Assume also that $R_\Phi\colon \A\da \A'$ is a Courant
morphism. Our goal in this Section is to formulate a sufficient
condition on the two splittings such that the map $\Phi\colon M\to M'$
is a bivector map,
\[ \pi\sim_\Phi\pi',\]
i.e. $(\d\Phi)_m(\pi_m)=\pi'_{\Phi(m)}$. 

Since the condition will only involve linear algebra we will
temporarily just deal with vector spaces $W$ with split bilinear form
$\l\cdot,\cdot\r$.  For instance, if $V$ is any vector space then
$\V=V\oplus V^*$ carries a natural split bilinear form given by the
pairing.  Suppose $W_1\subset W$ is a co-isotropic subspace, and
$W_0=W_1^\perp$ its orthogonal. Then $W_\red=W_1/W_0$ is again a
vector space with split bilinear form. For any Lagrangian subspace
$E\subset W$, the quotient $E_\red=E_1/E_0$ with $E_i=E\cap W_i$ is a
Lagrangian subspace of $E_\red$.

A bivector $\Pi\in\wedge^2(W)$ descends to a bivector $\Pi_\red$ on
$W_1/W_0$ if and only if it lies in the subspace
$\wedge^2(W_1)+W_0\wedge W$, or equivalently
\begin{equation}\label{eq:picond}
 w\in W_1\Rightarrow \iota(w)\Pi\in W_1.
\end{equation}
Here $\iota(w)\colon\wedge W\to \wedge W$
is the derivation extension of $w_1\mapsto \l w_1,w\r$. Note that 
\eqref{eq:picond} implies $w\in W_0\Rightarrow \iota(w)\Pi\in W_0$. 
Letting $w_\red\in W_\red$ be the image of $w\in W_1$, the bivector 
$\Pi_\red\in\wedge^2 W_\red$ is then given by 
\[ \iota(w_\red)\Pi_\red=(\iota(w)\Pi)_\red.\]
\begin{lemma}\label{lem:pilem}
  Suppose that $W=E\oplus F$ is a Lagrangian splitting, and let
  $\Pi\in\wedge^2 W$ be the bivector defined by
  $\iota_w\Pi=\hh(\pr_F(w)-\pr_E(w))$, i.e.
\begin{equation}\label{eq:piinthesky}
 \Pi=\hh \sum_i e_i\wedge f^i
\end{equation}
for dual bases $e_i,f^i$ of $E,F$. Then $\Pi$ descends to a bivector
$\Pi_\red$ on $W_\red$ if and only if $W_0=E_0\oplus F_0$, or
equivalently $W_1=E_1\oplus F_1$. Furthermore, in this
case $E_\red,\ F_\red$ are transverse, and $\Pi_\red$ is given by a 
formula similar to \eqref{eq:piinthesky} with dual bases for
$E_\red,\ F_\red$. 
\end{lemma}
\begin{proof}
Let $\pr_E,\ \pr_F$ denote the projections from $W$ to $E,F$. Thus 
$w=\pr_E(w)+\pr_F(w)$, while on the other hand
$\iota(w)\Pi=\hh(\pr_F(w)-\pr_E(w))$. 
Hence, the condition $w\in W_1\Rightarrow \iota(w)\Pi\in W_1$ holds if
and only if $w\in W_1\Rightarrow \pr_E(w),\pr_F(w)\in W_1$, that is
$W_1=E_1\oplus F_1$. Taking orthogonals, this is equivalent to
$W_0=(E+W_0)\cap (F+W_0)$. We claim that this in turn is equivalent to
$W_0=E_0\oplus F_0$. The implication $\Leftarrow$ is clear. For the
opposite implication $\Rightarrow$ let $w\in W_0$, and write $w=e+f$
with $e=\pr_E(w)$ and $f=\pr_F(w)$. Then $f=w-e\in (E+W_0)\cap
F= (E+W_0)\cap (F+W_0)\cap F=W_0\cap F=F_0$, and similarly $e\in E_0$.

Suppose then that $W_0=E_0\oplus F_0$. For $w\in W_1$, the
decomposition $w=x+y$ with $x=\pr_E(w)\in E_1,\ y=\pr_F(w)\in F_1$
descends to the decomposition $w_\red=x_\red+y_\red$ with $x\in
E_\red$ and $y\in F_\red$. Hence
\[ \iota(w_\red)\Pi_\red=
(\iota(w)\Pi)_\red=\hh(y-x)_\red=\hh (y_\red-x_\red)\]
which shows that $\Pi_\red$ is the bivector for the decomposition
$W_\red=E_\red\oplus F_\red$. 
\end{proof}
\begin{remark}
  The property $E_\red\cap F_\red=0$ is not automatic. Suppose e.g.
  that $w$ is an isotropic vector not contained in $E,F$, and let
  $W_0$ be the 1-dimensional subspace spanned by $w$.  Then
  $x=\pr_E(w)\in E_1$ and $y=-\pr_F(w)\in F_1$ descend to 
  non-zero elements $x_\red\in E_\red,\ y_\red\in F_\red$, 
  with $x_\red-y_\red=w_\red=0$. On the other hand $W_0=E_0\oplus
  F_0$ is not a necessary condition to ensure $E_\red\cap F_\red=0$, as
  one can see by taking $W_0$ to be any Lagrangian subspace transverse
  to $E,F$. (In this case the transversality condition is trivial since
  $W_\red=0$, while $E_0+F_0=0$.)
\end{remark}

A morphism of vector spaces $W,W'$ with split bilinear forms is a
Lagrangian subspace $R\subset W'\times \ol{W}$, where $\ol{W}$ is the
vector space $W$ with opposite bilinear form. As before we write
$R\colon W\da W'$, and 
$w\sim_R w'$ if $(w,w')\in R$. Let 
\begin{equation}\label{eq:wdef1}
\begin{split}
\ker(R)&=\{w\in W|\ w\sim_R 0\},\\ 
\ran(R)&=\{w\in W'|\ \exists\, w\in
W\colon\  w\sim_R w'\}.
\end{split}\end{equation}
The \emph{transpose} of the morphism $R\colon W\da W'$ is the morphism 
$R^t\colon W'\da W$, given by $w'\sim_{R^t} w\ \Leftrightarrow\ 
w\sim_R w'$.
\begin{lemma}\label{lem:isom}
We have $\ker(R)=\ran(R^t)^\perp$ and $\ran(R)=\ker(R^t)^\perp$. 
Furthermore, 
%
\begin{equation}\label{eq:dimensions}
\begin{split}
 \dim\ker(R)+\dim\ran(R)&=\dim R.\\
\end{split}
\end{equation}
The morphism $R\colon W\da W'$ gives a well-defined linear isomorphism
\[ R_\red\colon \ran(R^t)/\ker(R)\to \ran(R)/\ker(R^t)\]
with the property 
\[ w\sim_R w'\ \Rightarrow \ R_\red(w_\red)=w'_\red.\] 
Here $w_\red,w'_\red$ are the images of $w\in \ran(R^t)$ and 
$w'\in \ran(R)$ in the two quotient spaces. 
\end{lemma}
\begin{proof}
We will write 
\[ W_0=\ker(R),\ W_1=\ran(R^t),\ W_0'=\ker(R^t),\ W_1'=\ran(R)\]
  and $W_\red=W_1/W_0,\ W'_\red=W_1'/W_0'$.  The dimension formula
  \eqref{eq:dimensions} follows since the spaces $W_0$ and $W_1'$ are
  the kernel and range of the projection of $R$ onto $W'$. Taking the
  sum and the difference with a similar equation for $R^t$, we obtain 
\begin{equation}\label{eq:useful}
\begin{split}
& (\dim W_0+\dim W_1)+(\dim W_0'+\dim W_1')=\dim W+\dim W',\\
& \dim W_1-\dim W_0=\dim W_1'-\dim W_0'.
\end{split}\end{equation}
The second equation shows that $\dim W_\red=\dim W'_\red$.

Recall $w_i\sim_R w_i'$
  for $i=1,2$ implies $\l w_1,w_2\r=\l w_1',w_2'\r$. If $w_1\in W_0$ so that 
$w_1\sim 0$, and $w_2\in W_1$ so that $w_2\sim w_2'$ for some $w_2'\in W'$, 
this shows $\l w_1,w_2\r=\l 0,w_2'\r=0$. Thus $W_0\subset W_1^\perp$, 
and similarly $W_1'\subset (W_0')^\perp$. To prove equality
we observe that 
\begin{equation}\label{eq:ok}
\dim W_0+\dim W_1=\dim W,\ \ \dim W_0'+\dim W_1'=\dim W'.\end{equation}
Indeed we already know the inequality $\le$ in these equations, and
the equality follows by the first equation in \eqref{eq:useful}.
Finally, if $w\sim_R w'$, and $w-\ti{w}\in W_0$, $w'-\ti{w}'\in W_0'$,
then $\ti{w}\sim_R \ti{w}'$.  This shows that the map $R_\red\colon
W_\red\to W'_\red,\ w_\red\mapsto w'_\red$ is a well-defined isomorphism.
\end{proof}

\begin{definition}\label{def:lagrel}
  Given a morphism $R\colon W\da W'$ and Lagrangian subspaces
  $E\subset W,\ E'\subset W'$, we write
\begin{equation}\label{eq:lagrel}
 E\sim_R E'\end{equation}
if the isomorphism $R_\red$ from Lemma \ref{lem:isom} takes $E_\red$
to $E'_\red$. Given Lagrangian splittings $W=E\oplus F$ and
$W'=E'\oplus F'$, we write 
\begin{equation}\label{eq:lagrel2}
 (E,F)\sim_R (E',F')
\end{equation}
if $E\sim_{R} E',\ \ F\sim_{R} F'$, and in addition 
\begin{enumerate}
\item $\ker(R)$ is the direct sum of its intersections with
      $E$ and with $F$, 
\item $\ran(R)$ is the direct sum of its intersections with
      $E'$ and with $F'$.
\end{enumerate}
\end{definition}

\begin{remark}
  The relation \eqref{eq:lagrel} follows if $E=E'\circ R$ (b-Dirac
  morphism) or $E'=R\circ E$ (f-Dirac morphism). Conversely, if
  $\ker(R)=0$ the relation \eqref{eq:lagrel} implies $E=E'\circ R$,
  while for $\ran(R)=W'$ it implies $E'=R\circ E$. In general,
  \eqref{eq:lagrel} implies that $R$ is an f-Dirac morphism followed
  by a b-Dirac morphism:
\[ (W,E)\da (W'_\red,E'_\red)\da (W',E')\] 
with $W'_\red=\ran(R)/\ran(R)^\perp$. Note however that 
$E\sim_R E'$ and $E'\sim_S E''$ do not imply $E\sim_{S\circ
R}E''$ in general.
\end{remark}

\begin{proposition}
  Let $W=E\oplus F$, $W'=E'\oplus F'$ be Lagrangian splittings,
  defining bivectors $\Pi,\Pi'$, and $R\colon W\da W'$ a morphism with
  $(E,F)\sim_R (E',F')$. Then the induced map $R_\red\colon W_\red\to
  W'_\red$ takes $\Pi_\red$ to $\Pi'_\red$.
\end{proposition}
\begin{proof}
This follows since $\Pi_\red,\Pi'_\red$ are the bivectors defined by the splittings
$W_\red=E_\red\oplus F_\red$ and $W_\red'=E_\red'\oplus F_\red'$, and
since $R_\red(E_\red)=E_\red',\ R_\red(F_\red)=F_\red'$. 
\end{proof}

We now apply these results to Courant morphisms $R_\Phi\colon \A\da
\A'$ over manifolds $\Phi\colon M\to M'$. For Lagrangian splittings
$\A=E\oplus F$ and $\A'=E'\oplus F'$ we will write
\[ (E,F)\sim_{R_\Phi} (E',F')\]
if this relation holds pointwise, at any $m\in M$.  The main result of
this Section is as follows:


\begin{theorem}
  Let $\A=E\oplus F$ and $\A'=E'\oplus F'$ be Lagrangian splittings of
  Courant algebroids $\A,\ \A'$, defining bivectors
  $\pi\in\mf{X}^2(M),\ \pi'\in\mf{X}^2(M')$.
  Suppose 
\[ R_\Phi\colon \A\da \A'\] is a Courant morphism. Then 
\[ (E,F)\sim_{R_\Phi} (E',F')\  \Rightarrow\  
\pi\sim_\Phi \pi'.\]
\end{theorem}
\begin{proof}
  We have to show $\pi_{\Phi(m)}=(d\Phi)_m(\pi_m)$ for all $m\in M$. To
  simplify notation, we omit base points in the following discussion.
Let $\Pi,\Pi'$ be the bivectors defined by the splittings, so that
$\pi=\a(\Pi)$ and $\pi'=\a'(\Pi')$. We define 
\[ \A_\red=\ker(R_\Phi)^\perp/\ker(R_\Phi),\ \ \A'_\red=\ran(R_\Phi)/\ran(R_\Phi)^\perp,\]
so that $R_\Phi$ descends to an isomorphism $\A_\red\to \A'_\red$
taking $\Pi_\red$ to $\Pi'_\red$.  Since $x\sim_{R_\Phi} x'$ implies
$(\d\Phi)(\a(x))=\a'(x')$, we see that $\a'(\ker(R_\Phi)^\perp)=0$ and
$\a'(\ran(R_\Phi))\subset \on{ran}(\d\Phi)$.  Similarly
$\a(\ker(R_\Phi))\subset \ker(\d\Phi)$. Hence $\a,\a'$ descend to
define the vertical maps of the following commutative diagram,
\[\begin{CD} \A_\red @>>{\cong}> \A'_\red\\
@VV{\a_\red}V @VV{\a'_\red}V \\
TM/\ker(\d\Phi) @>>{\cong}>  \on{ran}(\d\Phi)
\end{CD}\]
The bottom map takes $\pi_\red=\a_\red(\Pi_\red)$ to
$\pi'_\red=\a'_\red(\Pi'_\red)$.  On the other hand, Lemma
\ref{lem:pilem} shows that $\pi'$ is the image of $\pi'_\red$ under
the inclusion $\on{ran}(\d\Phi)\to TM'$, while $\pi_\red$ is the image
of $\pi$ under the projection $TM\to TM/\ker(\d\Phi)$. This proves
that $\d\Phi\colon TM\to TM'$ takes $\pi$ to $\pi'$.
\end{proof}

\section{Courant algebroids of the form $M\times\g$}\label{sec:cou}
An \emph{action of a Lie algebra $\g$} on a manifold $M$ is a Lie
algebra homomorphism $\g\to \mf{X}(M)$. The range of the action map
$\a\colon M\times\g\to TM$ defines an integrable generalized
distribution; its leaves are called the \emph{$\g$-orbits}. The action
map makes $M\times\g$ into a Lie algebroid over $M$, with anchor map
$\a$ and Lie bracket $[\cdot,\cdot]_\mf{X}$ extending the Lie bracket
on constant sections. Explicitly,
\begin{equation}\label{eq:liealgebroid}
 [x_1,x_2]_{\mf{X}}=[x_1,x_2]_\g+\L_{\a(x_1)}x_2-\L_{\a(x_2)}x_1
\end{equation}
where $[x_1,x_2]_\g$ is the pointwise bracket. 
\subsection{$\g$-actions}\label{subsec:cou}
Let $\g$ be a quadratic Lie algebra, with inner product
$\l\cdot,\cdot\r$. Given a $\g$-action on a manifold $M$, all of
whose stabilizer are co-isotropic, Example \ref{ex:important}
describes a Courant algebroid structure on the product $\A_M=M\times\g$,
with anchor map $\a\colon M\times\g\to TM$ given by the action. 
Recall that the inner product $\l\cdot,\cdot\r$ and Courant bracket
$\Cour{\cdot,\cdot}$ on $\A_M$ extend the inner product and Lie
bracket on constant sections
$\g\subset \Gamma(\A_M)$. 

\begin{lemma}
The Courant bracket and the Lie algebroid bracket on
$\Gamma(\A_M)=C^\infty(M,\g)$ are related as follows: 
\begin{equation}\label{eq:formula}
 \Cour{x,y}=[x,y]_\X+\a^*\l \d x,y\r, \ \ x,y\in C^\infty(M,\g).
\end{equation}
\end{lemma}
\begin{proof}
  The Courant bracket is uniquely determined by its property of
  extending the Lie bracket on constant sections.  It hence suffices
  to note that the right hand side restricts to the pointwise bracket
  on constant sections, and that it satisfies c3) and p1) from
  \ref{subsec:def}.
\end{proof}

\subsection{Examples}
\begin{example} \label{ex:g}
  Let $\g$ be a quadratic Lie algebra, and $\ol{\g}$ the same Lie algebra with the opposite
  bilinear form. Then $\dd=\g\oplus \ol{\g}$ is a quadratic
  Lie algebra.  Let $G$ be a Lie group with Lie
  algebra $\g$, and denote by $u^L$ resp. $u^R$ the left-invariant
  resp.  right-invariant vector fields on $G$ defined by $u\in\g$. 
The action of $D=G\times G$ on $G$, given by $(a,b).g=agb^{-1}$,
defines an action of $\dd$, 
\[ \a\colon\dd\to TG,\ (u,v)\mapsto v^L-u^R\]
The stabilizers for the infinitesimal action are co-isotropic: In
fact,
\[ \ker(\a_g)=\dd_g=\{(u,v)|\, \ u=\Ad_g(v)\}\] 
is Lagrangian. Thus $\A_G=G\times\dd$ is a Courant algebroid. This is
extensively studied in \cite{al:pur}, where an explicit isomorphism
$\A_G\cong \TG^{(\eta)}$ is constructed (cf. Definition
\ref{def:standard}), with
$\eta=\f{1}{12}\theta^L\cdot[\theta^L,\theta^L]\in \Om^3(G)$ the
Cartan 3-form (where $\theta^L\in\Om^1(G,\g)$ is the left-invariant
Maurer-Cartan form.
\end{example}

\begin{example}[De Concini-Procesi compactification]
  Let $G$ be a complex semisimple Lie group of adjoint type, and
  $\dd=\g\oplus \ol{\g}$ as above. Let $M$ be the complex manifold
  given as its de Concini-Procesi `wonderful compactification'
  \cite{ev:won}. The action of $D=G\times G$ on $G$ 
(given by $(a,b).g=agb^{-1}$) extends to an
  action on $M$, and $\A_M=M\times\dd$ is a Courant algebroid. (The
  fact that elements in $\ran(\a^*)$ are isotropic follows by
  continuity.)
\end{example}

\begin{example}[Homogeneous spaces $G/H$]\label{ex:homogeneous}
  Let $\g$ be a quadratic Lie algebra, with corresponding Lie group
  $G$, and suppose that $H\subset G$ is a closed subgroup whose Lie
  algebra $\h$ is coisotropic. Then the stabilizer algebras for the
  induced $G$-action on $G/H$ are coisotropic. (The stabilizer algebra
  at the coset of $g\in G$ is $\Ad_g(\h)$.) Hence we obtain a Courant
  algebroid $\A_{G/H}=G/H\times\g$. In case $\h$ is a Lagrangian
  subalgebra, this example is discussed by P. {\v{S}}evera in
  \cite[\#2]{sev:let} and by Alekseev-Xu in \cite{al:der}, see also
  \cite{bur:cou}.
\end{example}

\begin{example}[The variety of Lagrangian subalgebras]
  A quadratic Lie algebra $\dd$ will be called \emph{split quadratic}
  if its bilinear form is split, i.e. if there exist Lagrangian
  subspaces.  For instance the Lie algebra $\g\oplus\ol{\g}$ from
  Example \ref{ex:g} is split quadratic, as is the semi-direct product
  $\k^*\rtimes\k$ for any Lie algebra $\k$. Given a split quadratic
  Lie algebra $\dd$ let $\on{Lag}(\dd)$ be the manifold of Lagrangian
  subspaces of $\dd$, and $M\subset \on{Lag}(\dd)$ the subset of
  Lagrangian Lie subalgebras. Then $M$ is not a manifold, but (working
  over $\C$) it is a (complex) variety.  Let $D$ be a Lie group
  exponentiating $\dd$, with its natural action on $M$. Then the
  stabilizer algebra at $m\in M$ contains the Lagrangian subalgebra
  labelled by $m$, and hence is co-isotropic.  Consequently
  $\A_M=M\times\dd$ is a Courant algebroid. A similar argument applies
  more generally to the \emph{variety of co-isotropic subalgebras} of
  any given dimension.
\end{example}

These examples are related: For instance, as shown by Evens-Lu
\cite{ev:on} the de Concini-Procesi compactification of $G$ is one of
the irreducible components of the variety of Lagrangian subalgebras of
$\dd=\g\oplus \ol{\g}$.  Example \ref{ex:g} is a special case of
\ref{ex:homogeneous}, taking the quotient of $D=G\times G$ by the
diagonal subgroup.

\begin{example}
Given a manifold $N$ with an action of $\g$, define a section $\gamma$
of the bundle $S^2(TN)$ by $\gamma(\mu,\nu)=\l\a^*(\mu),\a^*(\nu)\r$.
Then $M=\gamma^{-1}(0)$ is $\g$-invariant, and (assuming for
simplicity that $\gamma^{-1}(0)$ is smooth) $\A_M=M\times\g$ is a Courant
algebroid.
\end{example}

\subsection{Courant morphisms}\label{subsec{courmo}}
We will now consider morphisms between Courant algebroids of the form
$M\times\g$. 
\begin{proposition}\label{prop:generalpicture}
  Suppose $\g,\g'$ are quadratic Lie algebras, acting on $M,M'$ 
with co-isotropic stabilizers. Let $\A_M=M\times\g,\ \A_{M'}=M'\times\g'$ be
  the resulting Courant algebroids, with anchor map the action maps
  $\a,\a'$. Assume $\Phi\colon M\to M'$ is a
  smooth map and $R\subset \g'\oplus\ol{\g}$ is a Lagrangian
  subalgebra, with the property
\begin{equation}\label{eq:equivariance}
 x\sim_R x'\ \Rightarrow\ \a(x)\sim_\Phi \a'(x').
\end{equation}
Then $R_\Phi=\on{Graph}(\Phi)\times R$ defines a Courant morphism
\[ R_\Phi\colon \A_M \da \A_{M'},\ \ (m,x)\sim_{R_\Phi} (\Phi(m),x').\]
If $\g=\dd,\ \g'=\dd'$ are split quadratic, then 
Lagrangian splittings $\dd=\mf{e}\oplus \mf{f},\
\dd'=\mf{e}'\oplus \mf{f}'$ with 
$(\mf{e},\mf{f})\sim_R (\mf{e}',\mf{f}')$ give 
$R_\Phi$-related Lagrangian splittings of the Courant algebroids, 
\begin{equation}\label{eq:rs}
 (E,F)\sim_{R_\Phi} (E',F')\end{equation}
with $E=M\times\mf{e},\ F=M\times\mf{f},\ E'=M\times\mf{e}',\ 
F'=M\times\mf{f}'$. Hence the corresponding bivector fields are
$\Phi$-related: $\pi\sim_\Phi\pi'$.
\end{proposition}
\begin{proof}
  Since $R$ is a Lagrangian subalgebra, the direct product $(M'\times
  M)\times R\subset \A_M\times\ol{\A_{M'}}$ is a Dirac structure.
  Condition \eqref{eq:equivariance} guarantees that $(\a\times\a')(R)$
  is tangent to the graph of $\Phi$.  The statements about Lagrangian
  splittings are obvious since the relation \eqref{eq:rs} is defined
  pointwise.
\end{proof}

We stress that the relation $R\colon \g\da \g'$ need not
come from a Lie algebra homomorphism.  

\begin{remark}
  The Proposition \ref{prop:generalpicture} may be generalized to the
  set-up from Section \ref{subsec:pull}, replacing $R\colon \g\da
  \g'$ with a Courant morphism $\A\da \A'$.
\end{remark}

\begin{example}[The Courant algebroid $\A_G$]\label{ex:courmult}
Let $\g$ be a quadratic Lie algebra, with invariant inner product
$B_\g$.  The space $\dd=\g\oplus \ol{\g}$ may be viewed as a pair
groupoid over $\g$, with source and target map $s(x,y)=y,\ t(x,y)=x$
and with groupoid multiplication 
\[ z=z'\circ z''\ \  \mbox{ if } \ \ s(z)=s(z''),\
t(z)=t(z'),\ s(z')=t(z'').\]
Let  
\[ R\colon \dd\oplus \dd\da \dd,\ (z',z'')\sim_R z'\circ z''.\]
In more detail, $R\subset \dd\oplus \ol{\dd\oplus\dd}$ is the subspace
of elements $(z,z',z'')$ such that $z=z'\circ z''$.  Then $R$ is a
Lagrangian subalgebra, with $\ran(R)=\dd$ and $\ker(R)=\{(z',z'')|\ 
z'\circ z''=0\}$. Let $G$ be a Lie group integrating $\g$, with the
$\dd$-action $\a\colon \dd \to \mf{X}(G)$ and Courant algebroid $\A_G$
from Example \ref{ex:g}.  Let $\Mult\colon G\times G\to G$ be group
multiplication. One easily checks (cf. \cite[Proposition 3.8]{al:pur})
that
\[ z=z'\circ z''\ \ \Rightarrow \ \ (\a(z'),\a(z''))\sim_{\Mult}
\a(z).\]
Hence the conditions from Proposition \ref{prop:generalpicture} hold,
and we obtain a Courant morphism
\begin{equation}
 R_{\Mult}\colon \A_G\times\A_G\da \A_G.
\end{equation}
There is also a Courant morphism $U_\Inv\colon \A_G\da \ol{\A_G}$ lifting the
inversion map $\on{Inv}\colon G\to G,\ g\mapsto g^{-1}$.  It is given
by the direct product of $\on{Graph}(\on{Inv})\times U$, where 
\[ U\colon\dd\to \ol{\dd},\ \  z\sim_U z^{-1}\] 
with $z^{-1}=(v,u)$ the groupoid inverse of $z=(u,v)$.  The
associativity property of group multiplication, and the property
$\on{Inv}\circ \on{Mult}=\Mult\circ (\Inv\times\Inv)$ lift to the
Courant morphisms. Thus $G$ with the Courant algebroid $\A_G$ could justifiably be called a
\emph{Courant Lie group}.
%
\end{example}

\begin{example}[The action map $G\times M\to M$] \label{ex:action}
  Let $\g$ be a quadratic Lie algebra, $G$ a corresponding Lie group
  and $\A_G$ the Courant algebroid from Example \ref{ex:g}. Suppose
  $G$ acts on a manifold $M$, with co-isotropic stabilizer algebras.
  We claim that the action map $\Phi\colon G\times M\to M$ lifts to a
  morphism of Courant algebroids
\[ S_\Phi\colon \A_G\times\A_M\da \A_M.\]
The relation 
\[ S\colon \dd\oplus\g\da \g,\ \ ((z,z'),z')\sim_S z.\]
defines a Lagrangian subalgebra of $\g\oplus
\ol{\dd\oplus\g}$, and it has the required property
\eqref{eq:equivariance}. The action property $\Phi\circ
(\Mult_G\times\on{id}_M)=\Phi\ \circ\ (\on{id}_G\times\Phi)$ lifts to
a similar property of the Courant morphisms, so that one might call
this a \emph{Courant Lie group action}.
\end{example}

\begin{example}\label{ex:dg}
Let $\g$ be a quadratic Lie algebra, with corresponding Lie group
$G$. Suppose $H\subset G$ is a closed subgroup whose Lie algebra
$\h\subset\g$ is Lagrangian, and let $\A_{G/H}=G/H\times\g$ be the resulting
Courant algebroid. Then the quotient map $\Psi\colon G\to G/H$ lifts to a
Courant morphism
\[ T_\Psi\colon \A_G \da \A_{G/H},\]
by taking 
\[ T\colon \dd\da \g,\ \ (z,u)\sim_T z, \ z\in \g,\ u\in\h.\]
If the subalgebra $\h\subset\g$ is only co-isotropic, but contains a
Lagrangian subalgebra $\k\subset\h$, one obtains a Courant morphism
$T_\Psi$ by replacing $u\in\h$ with $u\in\k$ in the definition of
$T$.
\end{example}

\section{Lagrangian splittings}\label{sec:split}
Throughout this Section, we assume that $\dd$ is a split quadratic Lie
algebra. 

\subsection{Lagrangian splittings of $\A_M$}
Let $\A_M=M\times\dd$ be the Courant algebroid defined by a
$\dd$-action with co-isotropic stabilizers. For any Lagrangian
subspace $\mf{l}\subset\dd$, the subbundle $E=M\times\mf{l}$ is
Lagrangian. Let $\Upsilon^{\mf{l}}\in\wedge^3\mf{l}^*$ be defined by
$\Upsilon^{\mf{l}}(u_1,u_2,u_3)=\l u_1,[u_2,u_3]\r$, $u_i\in \mf{l}$.
Then $\Upsilon^{\mf{l}}$ vanishes if and only if $\mf{l}$ is a Lie
subalgebra.

\begin{proposition}
  The Courant tensor of the Lagrangian subbundle $E=M\times\mf{l}$ is 
  $\Upsilon^{\mf{l}}$ viewed as a constant section. In particular, it
  vanishes if and only if $\mf{l}$ is a Lagrangian subalgebra.
\end{proposition}
\begin{proof}
  This is immediate from the definition of $\Upsilon^E$, since the
  Courant bracket on constant sections coincides with the pointwise
  Lie bracket.
\end{proof}

Hence, if $(\dd,\g)$ is a Manin pair (i.e. $\g$ is a Lagrangian Lie
subalgebra), then $(\A_M,\ M\times\g)$ is a Manin pair over $M$. As a special case of Theorem
\ref{th:splittings} we obtain: 

\begin{theorem}\label{th:easy}
Suppose $\dd$ acts on $M$ with co-isotropic stabilizers.
  Let $\dd=\mf{e}\oplus \mf{f}$ be a decomposition into two
  Lagrangian subspaces, and let $E=M\times\mf{e},\ F=M\times\mf{f}$. 
Define a bi-vector field on $M$ by
\[ \pi=\hh \sum_i \a(e_i)\wedge \a(f^i)\]
where $e_i,f^i$ are dual bases of $\mf{e},\mf{f}$. Then 
\[\hh[\pi,\pi]=\mf{a}(\Upsilon^{\mf{e}})+\mf{a}(\Upsilon^{\mf{f}}).\]
The rank of $\pi$ at $m\in M$ is given by 
\begin{equation}\label{eq:rk1} \rk(\pi_m)=\dim(\a_m(\g_2))-\dim(\mf{l}_m\cap \g_1)\end{equation}
with the Lagrangian subalgebra $\mf{l}_m=\ran(\a_m^*)+\ker(\a_m)\cap
\g_2\subset \dd$.  If $\mf{e}=\g_1$ and $\mf{f}=\g_2$ are Lagrangian
subalgebras, then $(\A_M,\ E,\ F)$ is a Manin triple over $M$ and
hence $\pi$ is a Poisson structure. The symplectic leaves of $\pi$ are
contained in the intersections of the $\g_1$-orbits and $\g_2$-orbits,
with equality if and only if
$\ker(\a)=\ran(\a^*)+(\ker(\a_1)\oplus\ker(\a_2))$. Here
$\a_i\colon\g_i\to TM$ are the restrictions of the action.
\end{theorem}

As remarked in the introduction, the Poisson structure $\pi$ on $M$ 
given by Theorem \ref{th:easy}, in the case that $\g_i$ are
Lagrangian sub-algebras, is due to Lu-Yakimov \cite{lu:reg}.

\begin{remark}
  Note that for $\g_2$ a Lagrangian subalgebra, $m\mapsto
  \mf{l}_m=\ran(\a_m^*)+\ker(\a_m)\cap \g_2$ gives a map from $M$ into
  the variety of Lagrangian subalgebras of $\dd$ -- a generalization
  of the \emph{Drinfeld map}.  In general, this map need not be
  smooth.
\end{remark}

\begin{remark}
  If the stabilizer algebras for the $\dd$-action are Lagrangian, then
  necessarily $\ker(\a)=\ran(\a^*)$. In this case, if
  $(\dd,\g_1,\g_2)$ is a Manin triple, the symplectic
  leaves of the Manin triple $(\A_M,\ E,\ F)$
  are the intersections of $\g_1$-orbits with $\g_2$-orbits.
\end{remark}

\begin{example}
If $\dd=\g\oplus\ol{\g}$ as in Example \ref{ex:g}, the diagonal 
$\g_\Delta\subset\dd=\g\oplus\ol{\g}$ is a Lagrangian subalgebra, and
the anti-diagonal $\g_{-\Delta}=\{(x,-x)|\ x\in\g\}$ is a Lagrangian
complement. The anti-diagonal is not a subalgebra unless $\g$ is
Abelian. In fact, letting $\Xi\in\wedge^3\g$ be the structure
constants tensor for $\g$, with normalization 
$\Xi(x_1,x_2,x_3)=\f{1}{4}B(x_1,[x_2,x_3])$, 
the Courant tensor of the anti-diagonal is 
\[\Upsilon^{\g_{-\Delta}}(\ti{x}_1,\ti{x}_2,\ti{x}_3)=\Xi(x_1,x_2,x_3),\
\ \ 
x_i\in\g,\ \ \ti{x}_i=\hh(x_i,-x_i) .\]
Given a $\dd$-action on a manifold $M$, with co-isotropic stabilizers,
the resulting Lagrangian splitting of the Courant algebroid $\A_M=M\times\dd$
defines a bi-vector field $\pi\in \mf{X}^2(M)$ with 
\[ \hh[\pi,\pi]=\a(\Xi).\]
Up to an irrelevant constant this is the definition of a
\emph{quasi-Poisson manifold} as in \cite{al:qu}. In particular, the
group $G$, and the variety of Lagrangian subalgebras in
$\dd=\g\oplus\ol{\g}$ carry natural quasi-Poisson structures.
\end{example}

\begin{example}
  Suppose $\g$ is complex semi-simple, with triangular decomposition
  $\g=\n_-\oplus \h\oplus \n_+$, and let $\dd=\g\oplus\ol{\g}$.  Then
  another choice of a complement to the diagonal $\g_\Delta$ is
\[ \mf{l}=\h_{-\Delta}+(\n_-\oplus \n_+)\subset \g\oplus \ol{\g}.\] 
Since $\mf{l}$ is a Lie subalgebra, the Manin triple
$(\dd,\g_\Delta,\mf{l})$ defines a Poisson structure on $M$.  The
Poisson structure on $G$ obtained in this way is due to
Semenov-Tian-Shansky \cite{se:dr}. Its extension to a Poisson
structure on the variety of Lagrangian subalgebras (and hence in
particular the de Concini-Procesi compactification) is due to Evens-Lu
\cite{ev:on}.
\end{example}

\begin{example}
  Let $(\dd,\g_1,\g_2)$ be a Manin triple, and $D$ a Lie group
  integrating $\dd$. Given a closed subgroup $Q\subset D$ whose Lie
  algebra $\mf{q}$ is co-isotropic in $D$ (cf. Example
  \ref{ex:homogeneous}). The resulting Poisson structure
  $\pi\in\mf{X}^2(D/Q)$ was studied in detail in the work of
  Lu-Yakimov \cite{lu:reg}. At the coset $[d]=dQ\in D/Q$, we have
  $\ker(\a_{[d]})=\Ad_d(\mf{q})$, hence
  $\ran(\a_{[d]}^*)=\Ad_d(\mf{q}^\perp)$. Furthermore, as shown in
  \cite[Proposition 2.5]{lu:reg} the space
  $\mf{l}_d=\Ad_d(\mf{q}^\perp)+ \Ad_d(\mf{q})\cap\g_2$ is the
  Lagrangian subalgebra given by the Drinfeld homomorphism for the
  Poisson homogeneous space $D/Q$.  Hence \eqref{eq:rk1} reduces to
  the following formula \cite[(2.11)]{lu:reg}
\[
\on{rank}(\pi_{[d]})=\dim(\a_{[d]}(\g_2))-\dim(\mf{l}_d\cap\g_1).\]
%
\end{example}

\subsection{Lagrangian splittings of $\A_D$}\label{sec:splitdd}
Suppose $\dd$ is a split quadratic Lie algebra, and
$D$ a Lie group integrating $\dd$. Then $\A_D=D\times(\dd\oplus \ol{\dd})$
is a Courant algebroid, as a special case of Example \ref{ex:g} (with
$\dd$ playing the role of $\g$ in that example). If $(\dd,\g_1,\g_2)$
is a Manin triple, we obtain two Manin triples for $\dd\oplus\ol{\dd}$:
\[ (\dd\oplus\ol{\dd},\ \mf{e}_+,\ \mf{f}_+),\ \ \ \ 
(\dd\oplus\ol{\dd},\ \mf{e}_-,\ \mf{f}_-)\]
where
\[ \mf{e}_+=\g_1\oplus \g_2,\ \ \mf{f}_+=\g_2\oplus\g_1,\ \
\mf{e}_-=\g_1\oplus \g_1,\ \ \mf{f}_-=\g_2\oplus\g_2.\] 
Letting $E_\pm=D\times \mf{e}_\pm,\ \ F_\pm=D\times\mf{f}_\pm$ both
$(D,E_+,F_+)$ and $(D,E_-,F_-)$ are Manin triples over $D$, defining
Poisson structures $\pi_+,\ \pi_-\in\mf{X}^2(D)$.  Since the Courant
algebroid $D\times(\dd\oplus\ol{\dd})$ is exact, Theorem \ref{th:easy}
gives a simple description of the symplectic leaves: Let $G_1,G_2$ be
subgroups of $D$ integrating $\g_1,\g_2$, then the symplectic leaves
of $\pi_+$ are the components of the intersections of the $G_1-G_2$
double cosets with the $G_2-G_1$ double cosets, while the leaves of
$\pi_-$ are the components of the intersections of the $G_1-G_1$
double cosets with the $G_2-G_2$ double cosets.  Note that the
symplectic leaf of the group unit $e\in D$ for the Poisson structure
$\pi_+$ is an open neighbourhood of $e$, whereas relative to the
Poisson structure $\pi_-$ the leaf of $e$ is a point. Hence only
$\pi_-$ is a possible candidate for a Poisson Lie group structure on
$D$.
Under the morphism $R\colon \dd\oplus \dd\da \dd$, we find 
\[ \begin{split}
(\mf{e}_-\times \mf{e}_-,\mf{f}_-\times \mf{f}_-)&\sim_{R} (\mf{e}_-,\mf{f}_-),\\
(\mf{e}_+\times \mf{f}_+,\mf{f}_+\times \mf{e}_+)&\sim_{R} (\mf{e}_-,\mf{f}_-),\\
(\mf{e}_+\times \mf{f}_-,\mf{f}_+\times \mf{e}_-)&\sim_{R} (\mf{e}_+,\mf{f}_+),\\
(\mf{e}_-\times \mf{e}_+,\mf{f}_-\times \mf{f}_+)&\sim_{R} (\mf{e}_+,\mf{f}_+)\\
\end{split}\]
Note in particular that in each case, $\ker(R)$ is a direct sum of
its intersections with the two Lagrangian subspaces, as required in
Definition \ref{def:lagrel}.

Hence, we obtain similar relations $(E_-\times E_-,F_-\times
F_-)\sim_{R_\Mult} (E_-,F_-)$ etc., and therefore the following
relations of the Poisson bivector fields: 
\[\begin{split}
\pi_-\times\pi_-&\sim_{\Mult}\pi_-\\
\pi_+\times(-\pi_+)&\sim_{\Mult}\pi_-\\
\pi_+\times(-\pi_-)&\sim_{\Mult}\pi_+\\
\pi_-\times\pi_+&\sim_{\Mult}\pi_+
\end{split}\]
In particular $(D,\pi_-)$ is a Poisson Lie group, and its action on
$(D,\pi_+)$ is a Poisson action.
\begin{remark}
  The Poisson Lie group $(D,\pi_-)$ is the well-known \emph{Drinfeld
    double} (of the Poisson Lie group $G_1$, see below), while
  $(D,\pi_+)$ is known as the \emph{Heisenberg double}. One has the
  explicit formulas $\pi_\pm=\mf{r}^R\pm \mf{r}^L$, where the r-matrix
  $\mf{r}\in \wedge^2\dd$ is given in dual bases of $\g_1,\g_2$ by
  $\mf{r}=\hh\sum_i e_i\wedge f^i$, and the superscripts $L,R$
  indicate the extensions to left, right-invariant vector fields. The
  description of the symplectic leaves of $(D,\pi_+)$, as
  intersections of double cosets, was obtained by Alekseev-Malkin
  \cite{al:sym} using a direct calculation.
\end{remark}

\begin{remark}
  Given a Manin triple $(\dd,\g_1,\g_2)$, one may also consider the
  Manin triples $(\dd\oplus\ol{\dd},\dd_\Delta,\mf{e}_+)$ and
  $(\dd\oplus\ol{\dd},\dd_\Delta,\mf{f}_+)$. These define yet
  other Poisson structures on $D$, whose symplectic leaves are the
  intersections of conjugacy classes in $D$ with $G_2-G_1$ double
  cosets, respectively the $G_1-G_2$ double cosets.
\end{remark}

Consider now a $D$-manifold $M$ as in Example \ref{ex:action}. 
Let $\pi\in\mf{X}^2(M)$ be the Poisson bivector defined by the Manin
triple $(\dd,\g_1,\g_2)$. Under the 
morphism $S\colon (\dd\oplus\ol{\dd})\oplus\dd\da \dd$ from that
example, 
\[\begin{split}
 (\mf{e}_+\times \g_2, \mf{f}_+\times\g_1)&\sim_S (\g_1,\g_2),\\ 
(\mf{e}_-\times \g_1, \mf{f}_-\times\g_2)&\sim_S (\g_1,\g_2).
\end{split}\]    
It follows that relative to the action map $\Phi\colon D\times M\to
M$,
\[ \begin{split}\pi_+\times(-\pi)&\sim_\Phi \pi,\\ \pi_-\times \pi&\sim\pi.\end{split}\]
In particular, the $(D,\pi_-)$-action is a Poisson Lie group action.
Finally, let $G\subset D$ be a closed subgroup such that
$\g\subset\dd$ is a Lagrangian subalgebra, and consider the quotient
map $\Psi\colon D\to D/G$ and the morphism $T\colon \dd\oplus \ol{\dd}\da
\dd$ as in Example \ref{ex:dg}. Let $D/G$ carry
the Poisson structure $\pi$ defined by a Manin triple $(\dd,\g_1,\g_2)$. 
Assuming that 
\begin{equation}\label{eq:need}
 \g=(\g\cap \g_1)\oplus (\g\cap\g_2)\end{equation}
(which is automatic if $\g=\g_1$ or $\g=\g_2$) one finds 
\[ \begin{split}(\mf{e}_+,\mf{f}_+)&\sim_T (\g_1,\g_2)\\
(\mf{e}_-,\mf{f}_-)&\sim_T (\g_1,\g_2)
\end{split}\]
and consequently $\pi_+\sim_\Psi\pi,\ \pi_-\sim_\Psi \pi$. Here
\eqref{eq:need} ensures that $\ker(T)=0\oplus\g$ is the direct sum of
its intersections with $\mf{e}_+,\,\mf{f}_+$, respectively with
$\mf{e}_-,\mf{f}_-$.

\section{Poisson Lie groups}\label{sec:poi}
In the last Section, we explained how the choice of a Manin triple
$(\dd,\g_1,\g_2)$ defines Poisson structures $\pi_\pm$ on the Lie
group $D$ corresponding to $\dd$, where $\pi_-$ is multiplicative.
Let $\Phi\colon G_1\to D$ be a map exponentiating the inclusion
$\g_1\to\dd$ to the level of Lie groups. In this Section, we show that
the pull-back Courant algebroid $\Phi^!\A_D$
has the form $G_1\times\ol{\dd}$ for a suitable $\dd$-action on $G_1$
(the \emph{dressing action}).  We fix the following notation: For any
Lie group $G$ and any $\xi\in C^\infty(G,\g)$, we denote by
$\xi^L\in\mf{X}(G)$ the vector field corresponding to $\xi$ under left
trivialization $G\times\g\cong TG$, and by $\xi^R$ the vector field
corresponding to $\xi$ under right trivialization. If $\xi\in\g$
(viewed as a constant function $G\to \g$), this agrees with our
earlier notation for left-invariant and right-invariant vector fields.

\subsection{The Courant algebroid $G_1\times\dd$}
Let $p_1,p_2\colon\dd\to \g_i$ be the projection to the two summands of
the Manin triple  $(\dd,\g_1,\g_2)$. 

\begin{theorem}\label{th:restrict}
Each of the following two maps $\dd\to \mf{X}(G_1)$
\[ \zeta\mapsto p_1(\Ad_{\Phi(g)}\zeta)^R,\ \ \ 
\zeta\mapsto -p_1(\Ad_{\Phi(g^{-1})}\zeta)^L\]
defines a Lie algebra action of $\dd$, with co-isotropic
stabilizers. Let $\A^R=G_1\times \dd,\ \A^L=G_1\times\dd$ be the
Courant algebroids for these actions. 
Then 
\[ \A^L\cong \ol{\A^R}\cong \Phi^!\A_D.\]
\end{theorem}
\begin{proof}
  Let $C^R$ be the subbundle of $\A_D$
  consisting of elements $x$ with $\a(x)$ tangent to the
  foliation $\g_1^R\subset TD$ spanned by the right-invariant vector
  fields $\xi^R,\ \xi\in\g_1$.  Then $C^R$ contains the Lagrangian
  subbundle $\ker\a$, and hence is co-isotropic. We will show that
  $C^R/(C^R)^\perp$ is a Courant algebroid $D\times\ol{\dd}$ defined
  by the following $\dd$-action $\zeta\mapsto p_1(\Ad_d\zeta)^R\in
  \mf{X}(D)$. The isomorphism $\ol{\A^R}\cong \Phi^!\A_D$ and the description of the resulting action on
  $G_1$ are then immediate. Since  $\a(x)=\a(u,v)=v^L-u^R=(\Ad_d
  v-u)^R$ for $x=(u,v)$, the fiber of $C^R$ at $d\in D$ is given by 
\[ C^R|_d=\{(u,v)\in\dd\oplus\ol{\dd}|\  p_2(\Ad_d v-u)=0\}.\]
We read off that $C^R$ has rank $\dim\dd+\dim\g_1$, and hence
$(C^R)^\perp$ has rank $\dim\g_1$.  If $x\in C^R$, i.e.
$\iota_{\a(x)}\theta^R\in\g_1$ where $\theta^R\in \Om^1(D,\dd)$ is the
right-invariant Maurer Cartan form, it follows that for all
$\xi\in\g_1$
\[ \l x,\a^*\l\theta^R,\xi\r\r=\l \iota_{\a(x)}\theta^R,\xi\r=0.\]
Hence $(C^R)^\perp$ contains all $\a^*\l\theta^R,\xi\r=-(\xi,\Ad_{d^{-1}}\xi)$ with $\xi\in\g_1$,
and for dimensional reasons such elements span all of
$(C^R)^\perp|_d$. Hence, $(C^R)^\perp|_d$ consists of all those
elements $(u,v)\in C^R|_d$ for which $u\in\g_1$ and $\Ad_dv-u=0$. 
The set of elements $(u,v)\in C^R|_d$ with $u\in\g_2$ is hence a
complement to $(C^R)^\perp|_d$ in $C^R|_d$. Note that for any element
of this form, $u=p_2(u)=p_2(\Ad_d v)$.  We conclude that the trivial bundle $D\times \ol{\dd}$, embedded in
$\A_D=D\times(\dd\oplus\ol{\dd})$ by the map
\begin{equation}\label{eq:dress1}
 \zeta\mapsto
\phi^R(\zeta)=(p_2(\Ad_d\zeta),\zeta),
\end{equation} 
is a complement to $(C^R)^\perp$ in $C^R$. 
To compute the Courant bracket of two such sections let
$\zeta,\ti{\zeta}\in\dd$ and put $u=p_2(\Ad_d\zeta),\ 
\ti{u}=p_2(\Ad_d\ti{\zeta})\in C^\infty(D,\dd)$. We have
\[ \begin{split}
  \Cour{\phi^R(\zeta),\phi^R(\ti{\zeta})}&=[\phi^R(\zeta),\phi^R(\ti{\zeta})]_\mf{X}
  +\a^*\l\d\phi^R(\zeta),\phi^R(\ti{\zeta)}\r\\
  &=\Big([u,\ti{u}]+\L_{\a(\phi^R(\zeta))}\ti{u}-\L_{\a(\phi^R(\ti{\zeta}))}u,\ 
  [\zeta,\ti{\zeta}]\Big) +\a^* \l \d u,\ti{u}\r.
\end{split}
\]
But $\l \d u,\ti{u}\r=0$ since $u,\ti{u}$ take values in the isotropic
subalgebra $\g_2$. Hence the Courant bracket takes on the form
$(?,[\zeta,\ti{\zeta}])$ where
`$?$' indicates some function with values in $\g_2$. Since this
Courant bracket is a section of $C^R$, we know without calculation
that $?= p_2(\Ad_d[\zeta,\ti{\zeta}])$. We have thus shown
\begin{equation}\label{eq:hom}
 \Cour{\phi^R(\zeta),\phi^R(\ti{\zeta})}=\phi^R([\zeta,\ti{\zeta}]).\end{equation}
Finally, the anchor map for $C^R/(C^R)^\perp\cong D\times\ol{\dd}$ is
obtained from
\[ \a(\phi^R(\zeta))=\zeta^L-(p_2(\Ad_d\zeta))^R
=p_1(\Ad_d\zeta)^R.\]
This gives the desired description of $C^R/(C^R)^\perp$. Similarly, to
prove $\A^L=\Phi^!\A_D$ we define $C^L$ to be
the set of all $x$ with $\a(x)\in \g_1^L$.  Arguing as above, one
finds that $C^L/(C^L)^\perp$ is isomorphic to $D\times\dd$, where the
isomorphism is defined by sections
$\phi^L(\zeta)=(\zeta,p_2(\Ad_{d^{-1}}\zeta))$.  Again, we find that
$\phi^L$ is a Lie algebra homomorphism relative to the Courant
bracket.  The anchor map for $D\times\dd$ is given by $\zeta\mapsto
-p_1(\Ad_{d^{-1}}\zeta)^L$.
\end{proof}

It is worthwhile to summarize the main properties of the sections
$\phi^L,\phi^R$:
\begin{proposition}
The two sections
\[ \phi^L,\phi^R\colon\dd\to \Gamma(\A_D)=C^\infty(D,\dd\oplus\ol{\dd})\]
given by $\phi^L(\zeta)=(\zeta,p_2(\Ad_{d^{-1}}\zeta))$ and
$\phi^R(\zeta)=(p_2(\Ad_d\zeta),\zeta)$ are both homomorphisms
relative to the Courant bracket. They satisfy
$\l\phi^L(\zeta),\phi^L(\zeta)\r=-\l\phi^R(\zeta),\phi^R(\zeta)\r=\l\zeta,\zeta\r$,
and their images under the anchor map define $\dd$-actions on $D$, 
\begin{equation}\label{eq:dress}
 \zeta\mapsto -p_1(\Ad_{d^{-1}}\zeta)^L\ \  \mbox{ resp.}\ \
 \zeta\mapsto p_1(\Ad_d\zeta)^R.
\end{equation}
One has, for all $\zeta,\xi\in\dd$, 
\[
\begin{split}
\Cour{\phi^L(\zeta),\a^*\l\theta^L,\xi\r}&=-\a^*\l\theta^L,\,[p_1(\Ad_{d^{-1}}\zeta),\xi]\r,\
\\ 
\Cour{\phi^R(\zeta),\a^*\l\theta^R,\xi\r}&=-\a^*\l\theta^R,\, [p_1(Ad_d\zeta),\xi]\r.\end{split}
\]
\end{proposition}
\begin{proof}
The Courant bracket $\Cour{\phi^L(\zeta),\a^*\l\theta^L,\xi\r}$ is
computed using p4), together with 
\[\L_{\a(\phi^L(\zeta))}\l\theta^L,\xi\r
=-\l \theta^L,[p_1(\Ad_{d^{-1}}\zeta),\xi]\r.\] The computation of
$\Cour{\phi^R(\zeta),\a^*\l\theta^R,\xi\r}$ is similar, and the other
properties of $\phi^L,\phi^R$ had already been established in the
proof of Theorem \ref{th:restrict}.
\end{proof}

\begin{remark}
  The fact that \eqref{eq:dress} defines $\dd$-actions holds true for
  any Lie algebra $\dd$, with Lie algebras $\g_1,\g_2$ such that
  $\dd=\g_1\oplus\g_2$ as vector spaces. If $\dd$ carries an invariant
  inner product for which $\g_2$ is co-isotropic, the $\dd$-space
  $G_1$ (for each of these actions) defines a Courant algebroid
  $G_1\times\dd$, and similarly $D\times\dd$.
\end{remark}

Let us denote $\A_{G_1}=\Phi^!\A_D$. From now on, we will work with
the identification 
\[ \A_{G_1}=\ol{\A^R}=G_1\times\ol{\dd}\]
from Theorem \ref{th:restrict}. By Proposition \ref{prop:pullbacks},
the map $\Phi\colon G_1\to D$ lifts to a Courant morphism
\[ P_\Phi\colon \A_{G_1}\da \A_D,\ \ 
(g,\zeta)\sim_{P_\Phi} (\Phi(g),\varphi(\xi,\zeta)),\ \ \zeta\in\dd,\ \xi\in\g_1\]
with
\begin{equation}\label{eq:varphi}
 \varphi(\xi,\zeta)=(p_2(\Ad_{\Phi(g)}\zeta)-\xi,\zeta-\Ad_{\Phi(g^{-1})}\xi).
\end{equation}

\subsection{Multiplicative properties}
As discussed in Example \ref{ex:courmult} the multiplication map
$\Mult$ for the group $D$ lifts to a Courant morphism $R_\Mult\colon
\A_D\times \A_D\da \A_D$. 
\begin{proposition}
The multiplication map $\Mult\colon G_1\times G_1\to G_1$ lifts to a
Courant morphism 
\begin{equation}\label{eq:q}
Q_\Mult\colon
\A_{G_1}\times \A_{G_1}\da \A_{G_1},\end{equation}
with $R_\Mult\circ (P_\Phi\times P_\Phi)= P_\Phi\circ Q_\Mult$.  For
all $g',g''\in G_1$ with product $g=g'g''$, the fiber
$Q_\Mult|_{(g,g',g'')}$ consists of all elements
$(\zeta,\zeta',\zeta'')\in \dd\oplus\ol{\dd}\oplus\ol{\dd}$ such that
\begin{equation}\label{eq:complicated}
\zeta=\zeta''+\Ad_{\Phi((g'')^{-1})}p_1(\zeta'),\ \ 
p_2(\Ad_{\Phi(g'')}\zeta'')=p_2(\zeta').
\end{equation}
At any given point $(g,g',g'')$, $\ran(Q_\Mult)=\ol{\dd}$ while 
\[ \ker(Q_\Mult)=\{(-\Ad_{\Phi((g'')^{-1})}\xi,\xi)|\ \xi\in\g_1\}.\]
\end{proposition}
\begin{proof}
We will obtain $Q_\Mult$ as the `pull-back' of $R_\Mult$, in
the sense that 
\begin{equation}\label{eq:smult}
 Q_\Mult=\Phi^*(R_\Mult\cap (C\times C\times C))/\Phi^*(R_\Mult\cap
 (C^\perp\times C^\perp\times C^\perp)).\
\end{equation}
with $C=C^R$ as in the proof of Theorem \ref{th:restrict}. Recall that
$R_\Mult$ consists of elements of the form $(x,z;x,y,y,z)\subset
\dd\oplus\ol{\dd}\oplus\ol{\dd\oplus\ol{\dd}\oplus\dd\oplus\ol{\dd}}$,
while $C$ consists of elements $(u,v)\in \dd\oplus\ol{\dd}$ with $p_2(\Ad_d v-u)=0$.
Taking the intersection of $R_\Mult$ with $C\times C\times C$ imposes
conditions, at $(d,d',d'')\in D\times D\times D$,
\begin{equation}\label{eq:three}
p_2(\Ad_d z-x)=0,\ \ p_2(\Ad_{d'}y-x)=0,\ \ p_2(\Ad_{d''}z-y)=0.\end{equation}
Since the quotient map
$C/C^\perp$ takes $(u,v)$ to $\zeta=v-\Ad_{d^{-1}}p_1(u)$, the quotient on
the right hand side of \eqref{eq:smult} consists of elements $(\zeta,\zeta',\zeta'')$
where 
\[ \zeta=z-\Ad_{d^{-1}}p_1(x),\ \ \zeta'=y-\Ad_{(d')^{-1}}p_1(x),\ \ \zeta''=z-\Ad_{(d'')^{-1}}p_1(y).\]
with $x,y,z\in\dd$ subject to the conditions \eqref{eq:three}. It is
straightforward to check that these are related by
\eqref{eq:complicated} if $d',d'',d$ are the images of
$g',g'',g=g'g''\in G_1$. Suppose now that $(\zeta',\zeta'')\in
\ker(Q_\Mult)$ (at a given point $(g',g'')\in G_1\times G_1$). The
first equation in \eqref{eq:complicated} (with $\zeta=0$) shows that
$\zeta''\in\g_1$, and then the second equation shows that
$\zeta'\in\g_1$ also, with
$\zeta''+\Ad_{\Phi((g'')^{-1})}p_1(\zeta')=0$. This gives the desired
description of $\ker(Q_\Mult)$, while $\ran(Q_\Mult)=\ol{\dd}$ follows
by dimension count using \eqref{eq:dimensions}. 
\end{proof}

\subsection{Lagrangian splittings}
Let $\mf{e}_-=\g_1\oplus\g_1,\ \mf{f}_-=\g_2\oplus\g_2$ be the
Lagrangian splitting of $\dd\oplus\ol{\dd}$ discussed in Section
\ref{sec:splitdd}.  It defines the Manin triple
$(\A_D,E_-,F_-)$ with $E_-=D\times\mf{e}_-$ and
$F_-=D\times\mf{f}_-$, and the corresponding Poisson structure
$\pi_-$. On the other hand, the splitting $\dd=\g_1\oplus\g_2$ defines
a Manin triple $(\A_{G_1},\ E,\ F)$ over $G_1$, with 
$E=G_1\times\g_1$ and $F=G_1\times\g_2$. Let $\pi\in\mf{X}^2(G_1)$ be
the corresponding Poisson structure.
\begin{proposition}
The backward images of $E_-,F_-$ under the Courant morphism
$P_\Phi$ are
$E,F$ respectively. 
In fact, 
\[ (E,F)\sim_{P_\Phi}(E_-,F_-).\] 
so that $\Phi\colon G_1\to D$ is a Poisson map: $\pi\sim_\Phi \pi_-$.
\end{proposition}
\begin{proof}
  For any Lagrangian subspace $\mf{l}\subset (\dd\oplus \ol{\dd})$,
  the backward image $\mf{l}\circ P_\Phi|_{(\Phi(g),g)}$ consists of
  all $\zeta\in \dd$ such that there exists $\xi\in\g_1$ such that
  $\varphi(\xi,\zeta)$ (cf. \eqref{eq:varphi}) lies in $\mf{l}$.  For
  $\mf{l}=\mf{e}_-=g_1\oplus\g_1$, these conditions say that
  $\zeta\in\g_1$. For $\mf{l}_-=\mf{f}_-=\g_2\oplus\g_2$, the
  conditions say that $\xi=0$ and $\zeta\in \g_2$. This gives the
  desired description of the backward image. It remains to show that
  $\ran(P_\Phi)$ is the direct sum of its intersections with
  $E_-,F_-$, or equivalently that $\pr_E(\ran(P_\Phi))\subset
  \ran(P_\Phi)$. By definition, $\ran(R_\Phi)$ consists of all
  elements $\varphi(\xi,\zeta)$ with $\xi\in\g_1,\ \zeta\in\dd$. For
  any such element, the projection to $E_-$ is given by
\[ \pr_E(\varphi(\xi,\zeta))=(-\xi,p_1(\zeta)-\Ad_{\Phi(g^{-1})}\xi)
=\varphi(\xi,p_1(\zeta))\]
since $p_2(\Ad_{\Phi(g)}p_1(\zeta))=0$. Thus
$\pr_E(\varphi(\xi,\zeta))\in \ran(R_\Phi)$. 
\end{proof}

\begin{remark}
  By contrast, the backward images of $E_+,F_+$ under $P_\Phi$ are
  \emph{both} equal to $E$.
\end{remark}

\begin{proposition}
The forward image of $E\times E$ under the Courant
morphism $Q_\Mult$ is $E$. Similarly the forward image of
$F\times F$ is $F$. Indeed one has
\[ (E\times E, F\times F)
\sim_{Q_\Mult} (E,\,F);\]
in particular $\Mult\colon G_1\times G_1\to G_1$ is a Poisson map. 
\end{proposition}
\begin{proof}
  By \eqref{eq:complicated}, we see that the forward image
  $Q_\Mult|_{(g,g',g'')}\circ \mf{l}$ (for $g',g''\in G_1,\ g=g'g''$)
  of any Lagrangian subspace $\mf{l}\subset \dd\oplus\dd$ consists of
  all $\zeta=\zeta''+\Ad_{\Phi((g'')^{-1})}p_1(\zeta')$, such that
  $(\zeta',\zeta'')\in \mf{l}$ satisfying the condition
  $p_2(\Ad_{\Phi(g'')}\zeta'')=p_2(\zeta')$.  If
  $\mf{l}=\g_1\oplus\g_1$ the extra condition is automatic, and we
  find that $\zeta\in\g_1$. If $\mf{l}=\g_2\oplus\g_2$ the formula for
  $\zeta$ simplifies to $\zeta=\zeta''\in\g_2$, while the extra
  condition fixes $\zeta'$ as $\zeta'=p_2(\Ad_{\Phi(g'')}\zeta'')$
  This shows $(E\times E, F\times F)\sim_{Q_\Mult} (E,\,F)$: Indeed
  the conditions (a),(b) from Definition \ref{def:lagrel} on
  $\ker(Q_\Mult),\ran(Q_\Mult)$ are automatic since
  $\ker(Q_\Mult)\subset E\times E$ and
  $\ran(Q_\Mult)=G_1\times\ol{\dd}$.
\end{proof}


\bibliographystyle{amsplain} 

\def\cprime{$'$} \def\polhk#1{\setbox0=\hbox{#1}{\ooalign{\hidewidth
  \lower1.5ex\hbox{`}\hidewidth\crcr\unhbox0}}} \def\cprime{$'$}
  \def\cprime{$'$} \def\polhk#1{\setbox0=\hbox{#1}{\ooalign{\hidewidth
  \lower1.5ex\hbox{`}\hidewidth\crcr\unhbox0}}} \def\cprime{$'$}
  \def\cprime{$'$}
\providecommand{\bysame}{\leavevmode\hbox to3em{\hrulefill}\thinspace}
\providecommand{\MR}{\relax\ifhmode\unskip\space\fi MR }
\providecommand{\MRhref}[2]{%
  \href{http://www.ams.org/mathscinet-getitem?mr=#1}{#2}
}
\providecommand{\href}[2]{#2}

\end{document}